\documentclass{article}
\usepackage{amsmath,amssymb,amsthm,graphicx,epsfig,float,url}
\usepackage[colorlinks=true]{hyperref}
\usepackage{pdfsync}
\usepackage[usenames, dvipsnames]{xcolor}
\usepackage{tikz}
\usepackage{subfig}
\usepackage{stmaryrd}

\newcommand{\R}{\textnormal{I\kern-0.21emR}}
\newcommand{\N}{\textnormal{I\kern-0.21emN}}

\renewcommand{\geq}{\geqslant}
\renewcommand{\leq}{\leqslant}

\newcommand{\ep}{\varepsilon}

\newtheorem{proposition}{Proposition}

\newtheorem{lemma}{Lemma}

\theoremstyle{definition}

\title{\bf Control strategies on mosquitos population for the fight against arboviruses}

\author{
	Luis Almeida\footnote{Sorbonne Universit\'e, CNRS, UPMC Univ Paris 06, UMR 7598, Laboratoire Jacques-Louis Lions, \'Equipe MAMBA, F-75005, Paris, France ({\tt luis.almeida@ljll.math.upmc.fr}).}
	\and Michel Duprez\footnote{Sorbonne Universit\'e, CNRS, UPMC Univ Paris 06, UMR 7598, Laboratoire Jacques-Louis Lions, \'Equipe-projet CAGE, F-75005, Paris, France ({\tt mduprez@math.cnrs.fr}).}
	\and Yannick Privat\footnote{IRMA, Universit\'e de Strasbourg, CNRS UMR 7501, \'Equipe TONUS, 7 rue Ren\'e Descartes, 67084 Strasbourg, France ({\tt yannick.privat@unistra.fr}).}
	\and Nicolas Vauchelet\footnote{LAGA - UMR 7539 Institut Galil\'{e}e Universit\'{e} Paris 13, 99 avenue Jean-Baptiste Cl\'{e}ment, 93430 Villetaneuse - France ({\tt vauchelet@math.univ-paris13.fr})}
}

\date{}

\begin{document}

\maketitle

\begin{abstract}
In the fight against vector-borne arboviruses, an important strategy of control of epidemic consists in controlling the population of vector, \textit{Aedes} mosquitoes in this case.
Among possible actions, two techniques consist in releasing mosquitoes to reduce the size of the population (Sterile Insect Technique) or in  replacing the wild population by a population carrying a bacteria, called \textit{Wolbachia}, blocking the transmission of viruses from mosquitoes to human.
This paper is concerned with the question of optimizing the release protocol for these two strategies with the aim of getting as close as possible to the objectives.
Starting from a mathematical model describing the dynamics of the population, we include the control function and introduce the cost functional for both \textit{population replacement} and \textit{Sterile Insect Technique} problems. Next, we establish some properties of the optimal control and illustrate them with some numerical simulations.
\end{abstract}


\medskip


\section{Introduction}

Due to the major world-wide impact of vector-borne diseases on human health, many strategies, integrating more or less the three main actors of transmission (pathogen, vectors, man) were developed to reduce their spread.
In this work, we are interested in studying strategies targeting only the vector (mosquito belonging to the genus Aedes) of viral diseases such as dengue, chikungunya, zika.
In this aim, mathematical modelling has an important role since it allows to study and design different scenarios.
In this work, we focus on two strategies : the sterile insect technique and the population replacement.

The sterile insect technique consists in a massive releasing into the wild of sterilized males to mate with females in the aim to reduce the size of the insect population. It has been first studied by R. Bushland and E. Knipling and experimented successfully in the early 1950's by nearly eradicating screw-worm fly in North America.
Since then, this technique has been studied on different pest and disease vectors \cite{Barclay,SIT}.
In particular, it is of interest for control of mosquito populations and has been modeled mathematically and studied in several papers, see e.g. \cite{Anguelov,Dumont1,Dumont2,Cai,LiYuan,Sallet,Huang,bossin,bliman3}.

Recently, there has been an increasing interest in the biology of {\itshape Wolbachia} and in its application as an agent for control of vector mosquito populations, by taking advantage of phenomena called {\itshape cytoplasmic incompatibility} (CI) and {\itshape pathogen interference} (PI) \cite{Bourtzis,Sinkins}.
In key vector species such as {\itshape Aedes aegypti}, if a male mosquito infected with {\itshape Wolbachia} mates with a non-infected female, the embryos die early in development, in the first mitotic divisions \cite{Wer.Wolbachia}. This is the {\itshape cytoplasmic incompatibility} (CI).
The {\itshape pathogen interference} (PI) is characterized by the disability for some {\itshape Wolbachia} strains to transmit viruses like dengue, chikungunya, zika viruses in {\itshape Aedes} mosquitoes \cite{Wal.wMel}.
Once released, they breed with wild mosquitoes. Over time and if releases are large and long enough, one can expect the majority of mosquitoes to carry {\itshape Wolbachia}, thanks to CI. Due to PI, the mosquito population then has a reduced vector competence, decreasing the risk of dengue, chikungunya, and zika outbreaks.
The technique consisting of releasing mosquitoes carrying {\itshape Wolbachia} to replace the wild population is called population replacement.
It has been modeled and studied in several works, see e.g. \cite{Farkas,Fenton,Schraiber,Hughes,bliman1,PhD}.

In this paper, we are interested in the study of the optimization of the release protocol. More precisely, given a duration of the experiment and an amount of mosquitoes, what should be the temporal distribution of releases to be as close as possible to the objective to be reached at the final time of the experiment ?
To answer to this question, we first define a cost functional which will represent mathematically the objective we seek to attain.
For the sterile insect technique, the goal being to decrease the size of the population, the quantity to minimize will be defined as the number of females at the final time. For the population replacement with \textit{Wolbachia}, the quantity to minimize will be the distance (in the least squares sense) at final time to the infected equilibrium, corresponding to the state where all the mosquitoes carry the bacterium \textit{Wolbachia} (the entire population is infected).
Obviously, this optimal problem should satisfy the constraint that the number of mosquitoes released during the experiment is bounded.
Similar optimization problems for sterile insect or population replacement techniques, with different cost functionals, have been proposed in e.g. \cite{Thome,colombien,bliman2}.
Compared to previous work, the main difference here is due to the fact that we only consider the state at the final time, which seems natural but induces several technical difficulties.

The outline of the paper is the following. In the next section, we describe the mathematical modelling for the two strategies. Starting from a model incorporating the whole mosquito life cycle, we use several assumptions to simplify this system and arrive to two simple systems modeling the two techniques studied in this article. In Section \ref{sec:optim}, we introduce the cost functionals and describe the optimization problems to be solved and give existence results and some properties of the optimal control. These results are illustrated in Section \ref{sec:num} where some numerical simulations are provided. Then, a conclusion and a discussion of our results conclude this work. Finally, an Appendix is devoted to some technical proofs, in particular the proofs of existence of an optimal control are provided in this Appendix.

\section{Mathematical modelling}

\subsection{Mosquito life cycle}

The life cycle of a mosquito (male or female) occurs successively in two distinct environments:
it includes an aquatic phase (egg, larva, pupa), and an adult aerial phase.
A few days after being fertilized, a female mosquito may deposit a few dozen eggs possibly divided between several breeding sites.
Once deposited, the eggs of some species can resist up to several months and also to adverse weather conditions before hatching. This characteristic contributes to the adaptability of mosquitoes, and allowed them to colonize temperate regions.
After stimulation (e.g. the rain), the eggs hatch to give rise to larvae that will develop in the water and reach the state of pupae. This larval phase can last from a few days to a few weeks. Then the insect makes its metamorphosis. The pupa (also called {\it nymph}) stays in the aquatic state for 1 to 3 days and then becomes an adult mosquito (or {\it imago}): this is the emergence, and then the beginning of the aerial phase. Roughly speaking, the lifespan of an adult mosquito is estimated to be a few weeks.

In many species, oviposition is possible only after a blood meal, that is, the female must bite a vertebrate before each egg-laying. This behavior, called hematophagy, can be exploited by infectious agents (such as bacteria, viruses or parasites) to spread, passing alternately between a vertebrate host (man, for what interests us here) and an arthropod host (here, the mosquito).

In order to model this life cycle dynamics, we introduce the following quantities:
\begin{itemize}
\item $E(t)$ density of eggs at time $t$;
\item $L(t)$ larvae density at time $t$;
\item $P(t)$ pupa density at time $t$;
\item $F(t)$ and $M(t)$ density of adult females and males, respectively, at time $t$.
\end{itemize}
Let us consider the parameters:
\begin{itemize}
\item $\beta_E>0$ is oviposition rate for females;
\item $\delta_E, \delta_L, \delta_P, \delta_F, \delta_M >0$ are death rates for eggs, larvae, pupa, adult females, and males, respectively;
\item $\tau_E$ hatching rate for eggs;
\item $\nu$ the probability that a pupa gives rise to a female, therefore $(1-\nu)$ is the probability to give rise to a male ($0<\nu<1$);
\item $\tau_L$ and $\tau_P >0$ transition rates from larval phase to pupa and from pupa to adult;
\item intraspecific competition is supposed to occur only at the aquatic phase.
  This models on the one hand the occupation of the breeding sites that can only accommodate a limited number of eggs and, on the other hand, limited access to resources for the larvae.
  In the larval compartment, this competition is described by the introduction of a positive constant denoted $c$ and is supposed to depend on the concentration of the larvae: the greater the number of larvae is, the more the competition to find the essential nutrients for larval maturation is important.
  The environmental capacity for eggs is denoted $K$. This amount can be interpreted as the maximum density of eggs that females can lay in breeding sites.
\end{itemize}

From the above considerations, we can determine the dynamics of the mosquito population and obtain the following dynamical system
\begin{equation}
  \left\{
    \begin{aligned}
\frac{d}{dt} E &= \beta_E F \left(1-\frac{E}{K}\right) - E \big( \tau_E + \delta_E \big),  \\
\frac{d}{dt} L &= \tau_E E - L \big( c L + \tau_L  + \delta_L \big),  \\
\frac{d}{dt} P &= \tau_L L - (\delta_P + \tau_P) P,  \\
\frac{d}{dt} F &= \nu \tau_P P - \delta_F F, \\
\frac{d}{dt} M &= (1-\nu)\tau_P P - \delta_M M.
\end{aligned}
  \right.
  \label{eq:S1}
\end{equation}
  
It is important to notice that this system is appropriate only for a large enough number of mosquitoes, since in this model it is assumed that a female will mate with a male with a probability equal to $1$.
This assumption seems reasonable for high number of mosquitoes and is done in several models \cite{Yang}.
In more generality, one may consider that the rate $\beta_E$ depends on $M$ as a function $\beta_E(M)$, which complexifies the study performed here.
We refer to \cite{bossin,postdoc} and references therein for examples of such function $\beta_E$.
In order to further simplify this system of ODE, we assume that the time dynamics of the pupa compartment is fast. Then, denoting $\tilde{t}=\ep t$ a new time variable, and $\tilde{P}(\tilde{t})=P(t)$, we have
$$
\frac{d}{d\tilde{t}} \tilde{P} = \ep \frac{d}{dt} P = \tau_L L - (\delta_P + \tau_P) P.
$$
As $\ep \to 0$, we deduce that we may replace the third equation in system \eqref{eq:S1} by
$$
0 = \tau_L L - \delta_P P - \tau_P P,
$$
which implies the relation
$P=\frac{\tau_L}{\delta_P+\tau_P} L$.

To reduce further this system of equations, we will use some assumptions on the larval compartment. We first consider that the competition at the larvae stage is negligible (i.e. $c\ll 1$). Moreover, in favorable conditions, the larval stage may be really fast. Then, by the same token as above, this compartment may be considered at equilibrium leading to the relation
$$
\tau_E E = (\tau_L + \delta_L) L.
$$
Injecting this relation, system \eqref{eq:S1} reduces to
\begin{equation}
  \left\{
      \begin{aligned}
\frac{d}{dt} E &= \beta_E F\left(1-\frac{E}{K}\right) - (\tau_E+\delta_E) E,  \\
\frac{d}{dt} F &= \nu \beta_F E - \delta_F F, \\
\frac{d}{dt} M &= (1-\nu) \beta_F E - \delta_M M, 
\end{aligned}
  \right.
  \label{eq:S2}
\end{equation}
where we use the notation $\beta_F =\frac{\tau_P\tau_L\tau_E}{(\delta_P+\tau_P)(\tau_L+\delta_L)}$.

\subsection{Sterile insect technique}

As explained above, the sterile insect technique consists in releasing sterile males to mate with females with the aim of reducing the size of the population. We denote by $M_s$ the density of sterile males.
Only females mating fertile males will be able to lay eggs.
Assuming a uniform repartition of the population of mosquitos, the probability that a female mates with a fertile male is given by $\frac{M}{M+\gamma M_s}$.
The parameter $\gamma$ account for the fact that female may have a preference for fertile male.
Introducing the sterile male population into system \eqref{eq:S2} leads to
\begin{equation}
  \left\{
  \begin{aligned}
\frac{d}{dt} E &= \beta_E F\left(1-\frac{E}{K}\right)\frac{M}{M+\gamma M_s} - (\tau_E+\delta_E) E,  \\
\frac{d}{dt} F &= \nu \beta_F E - \delta_F F, \\
\frac{d}{dt} M &= (1-\nu) \beta_F E - \delta_M M, \\
\frac{d}{dt} M_s &= - \delta_s M_s.
  \end{aligned}
  \right.
  \label{eq:sterile1}
\end{equation}
It is clear that the extinction state, where $E=F=M=M_s=0$ is a steady state. However, an important observation for this system is that this steady state cannot be reached. Indeed, under suitable assumptions on the parameters, it is unstable as stated in the following proposition (whose proof is postponed in the Appendix).
\begin{proposition}\label{prop:sanseffetAllee}
Assume that 
\begin{equation}\label{eq:cond prop}
\delta_s>\delta_M\qquad \text{and}\qquad
\nu\beta_E \beta_F > \delta_F(\tau_E+\delta_E).
\end{equation}
Then the steady state $(0,0,0,0)$ for system \eqref{eq:sterile1} is unstable.
\end{proposition}
The first assumption in \eqref{eq:cond prop} traduces the fact that the death rate for the released sterile mosquitoes is higher than for wild ones.
The second assumption implies that the oviposition rate is high enough.
Notice that with values taken from the field, these assumptions are satisfied (see Section \ref{sec:num}).

Due to the high number of equations in system \eqref{eq:sterile1}, we will reduce this system by making the following assumption : 
The death rate for males and females is the same ($\delta_F=\delta_M$) and the probability that a pupa emerges to a female or a male is the same ($\nu=\frac 12$). Thanks to this assumption, males and females densities satisfy the same equation. Hence assuming that initially these quantities are equals, we will assume that $F=M$.

Finally, system \eqref{eq:sterile1} reduces to
\begin{equation}
  \left\{
  \begin{aligned}
\frac{d}{dt} E &= \beta_E F\left(1-\frac{E}{K}\right)\frac{F}{F+\gamma M_s} - (\tau_E+\delta_E) E,  \\
\frac{d}{dt} F &= \nu\beta_F E - \delta_F F, \\
\frac{d}{dt} M_s &= - \delta_S M_s.
  \end{aligned}
  \right.
  \label{eq:sterile}
\end{equation}

\subsection{Introduction of the bacteria {\it Wolbachia}}

To model the strategy consisting of releasing {\it Wolbachia} infected mosquitoes to replace the wild population, we introduce the infected population into \eqref{eq:S2}.
Let us denote $E_i$, $F_i$, $M_i$ the eggs, females and males compartments infected by \textit{Wolbachia} and $E_u$, $F_u$, $M_u$ correspond to the uninfected compartments.

Assuming an uniform repartition of the population of mosquitoes, then the probability for a female to mate with an infected male is equal to the proportion of infected males into the population, i.e. $\frac{M_i}{M_u+M_i}$.
Similarly, the probability to mate with a uninfected male is $\frac{M_u}{M_u+M_i}$.
To model the \textit{cytoplasmic incompatibility}, we introduce
a parameter denoted $s_h$, corresponding to the fraction of uninfected females eggs fertilized by infected males which will not hatch.
We have $0<s_h\leq 1$, the case $s_h=1$ correspond to the perfect cytoplasmic incompatibility.
From system \eqref{eq:S2}, we construct the following system taking into account infected and uninfected mosquitos:
\begin{equation}
\left\{
\begin{aligned}
\frac{d}{dt} E_u &= \beta_E F_u\left(\frac{M_u}{M_u+M_i}+(1-s_h)\frac{M_i}{M_u+M_i}\right)\left(1-\frac{E_u+E_i}{K}\right) - (\tau_E+\delta_E) E_u,  \\
\frac{d}{dt} F_u &= \nu \beta_F E_u - \delta_{F} F_u, \\
\frac{d}{dt} M_u &= (1-\nu) \beta_F E_u - \delta_{M} M_u, \\
\frac{d}{dt} E_i &= \eta \beta_E F_i\left(1-\frac{E_u+E_i}{K}\right) - (\tau_E+\delta_E) E_i,  \\
\frac{d}{dt} F_i &= \nu \beta_F E_i - \delta \delta_{F} F_i, \\
\frac{d}{dt} M_i &= (1-\nu) \beta_F E_i - \delta \delta_{M} M_i.
\end{aligned}
\right.
\label{eq:S3}
\end{equation}
In this system, we have introduced the two following parameters:
$\eta<1$ modelling the fecundity reduction of infected females with respect to uninfected females,
$\delta>1$ modelling the increase of mortality for infected mosquitoes.

As above, for the sterile insect technique, we make use of the same set of assumptions ($\nu=\frac 12$ and $\delta_F=\delta_M$) to reduce the system by considering that the quantity of males and females is the same : $M_u=F_u$ and $M_i=F_i$.

Under these assumptions, system \eqref{eq:S3} for the {\it Wolbachia} strategy reduces to
\begin{equation}
  \left\{
      \begin{aligned}
\frac{d}{dt} E_u &= \beta_E F_u \left(1-s_h \frac{F_i}{F_u+F_i}\right)\left(1-\frac{E_u+E_i}{K}\right) - \big(\tau_E+\delta_E\big) E_u,  \\
\frac{d}{dt} F_u &= \nu \beta_F E_u - \delta_{F} F_u, \\
\frac{d}{dt} E_i &= \eta \beta_E F_i\left(1-\frac{E_u+E_i}{K}\right) - \big(\tau_E+\delta_E\big) E_i,  \\
\frac{d}{dt} F_i &= \nu \beta_F E_i - \delta \delta_{F} F_i.
\end{aligned}
    \right.
    \label{eq:S4}
\end{equation}

\section{Towards optimisation problems}
\label{sec:optim}

We introduce in this section the optimisation problems considered in this work.
Since for both strategies, the idea consists in releasing mosquitos (sterile male or infected by Wolbachia), the control is about the release function which will be denoted $u$.
We assume that the release occurs in a time interval $[0,T]$ for $T>0$ given.
Obviously some constraints should be satisfied by the release function.
We assume that there exists $C\geq 0$ and $\overline{U}\geq 0$ such that
$0\leq u\leq \overline{U}$ a.e. and $\int_0^T u(t)dt \leq C$.
The first bound means that $u$, the instantaneous rate of mosquito release (number of mosquitoes per unit of time) is bounded by a constant $\overline{U}$ all along the period $[0,T]$; the second means that the total number of mosquitoes released is bounded by another $C$. Both assumptions are natural considering that one cannot produce an infinite number of mosquitoes to release nor release them at an infinite rate.

Before the beginning of the experiment the systems are assumed to be at equilibrium.
Hence for each system we first determine the equilibria.
We present successively the optimisation problem considered for the sterile insect technique and for the population replacement by {\itshape{Wolbachia}}.

\subsection{Sterile insect technique}

Let us first consider the system \eqref{eq:sterile} for the sterile insect technique. The following lemma gives the equilibria.
\begin{lemma}
  Under the assumption \eqref{eq:cond prop}, there are two equilibria for system \eqref{eq:sterile}:
  the extinction equilibria given by $(E_1^*,F_1^*,M_s^*)=(0,0,0)$,
  and the non-extinction equilibria $(E_2^*,F_2^*,M_s^*)=\left(\bar E_2, \frac{\nu \beta_F}{\delta_F}\bar E_2,0\right)$ with $\bar E_2=K\left(1-\frac{(\tau_E+\delta_E)\delta_F}{\nu \beta_E \beta_F}\right)$.
  Moreover, the non-extinction equilibrium is linearly asymptotically stable.
\end{lemma}

Let us denote $u$ the release function of sterile male mosquitoes.
Then system \eqref{eq:sterile} reads
\begin{equation}
  \left\{
  \begin{aligned}
&\frac{d}{dt} E = \beta_E F\left(1-\frac{E}{K}\right)\frac{F}{F+\gamma M_s} - (\tau_E+\delta_E) E,  \\
&\frac{d}{dt} F = \nu \beta_F E - \delta_F F, \\
&\frac{d}{dt} M_s = u - \delta_s M_s,  \\
& E(0) = E_2^*, \qquad F(0) = F_2^*, \qquad M_s(0)=0.
  \end{aligned}
  \right.
  \label{pb:sterile}
\end{equation}
In our minimisation problem, we want to find the release function $u$ under the above mentionned physical constraints for which the solution of the final time is the closest as possible of the extinction equilibrium.
More precisely, let us introduce the cost functional
\begin{equation*}
J(u) = \frac 12 \left(E(T)^2 + F(T)^2\right).
\end{equation*}
We want to solve the problem
\begin{equation}\label{pb_TIS}
\min_{u\in {\mathcal U}_{C,\overline{U}}} J(u), \qquad
{\mathcal U}_{C,\overline{U}} = \left\{0\leq u\leq \overline{U}, \quad \int_0^T u(t)dt \leq C\right\}.
\end{equation}

The following result, whose proof is postponed in the Appendix, gives the existence of a solution to this problem.
\begin{proposition}\label{prop:existsterile}
Under the assumption \eqref{eq:cond prop}, problem \eqref{pb_TIS} has a solution $u^*$. Moreover, assuming that
\begin{equation}\label{cond:barUT}
\overline{U}T>C,
\end{equation}
the optimal control strategy uses the maximal amount of mosquitoes, in other words  
$$
\int_0^T u^*(t)\, dt=C
$$
and there exists $T_0\in (0,T)$ such that $u^*=0$ on $(T_0,T)$. 
\end{proposition}

\subsection{Population replacement}

Let us consider the reduced model \eqref{eq:S4} for the introduction of the bacteria {\it Wolbachia}. The following Lemma gives the equilibria for this system:
\begin{lemma}\label{lem:eqWol}
Let us consider that $1<\delta$, $\eta<1$, $0<s_h\leq 1$. We denote $\displaystyle b=\frac{\nu\beta_F\beta_E}{(\tau_E+\delta_E)}$. Assume moreover that 
\begin{equation}\label{hyp:coeffW}
\eta b > \delta \delta_F, \qquad \frac{\eta}{\delta^2} < K\left(1-\frac{\delta\delta_F}{\eta b}\right) < \frac{\eta}{\delta(1-s_h)}.
\end{equation}
Then there are four distinct nonnegative equilibria:
\begin{itemize}
\item \textit{Wolbachia} invasion 
$\displaystyle (E^*_{uW}, F^*_{uW}, E^*_{iW}, F^*_{iW}) := \left(0,0,K\Big(1-\frac{\delta\delta_F}{b\eta}\Big),K\Big(\frac{\nu\beta_F}{\delta\delta_F}-\frac{\nu\beta_F}{b\eta}\Big)\right)$
is stable;
\item \textit{Wolbachia} extinction
$\displaystyle (E^*_{uE}, F^*_{uE}, E^*_{iE}, F^*_{iE}) := \left(K\Big(1-\frac{\delta_F}{b}\Big),K\Big(\frac{\nu\beta_F}{\delta_F}-\frac{\nu\beta_F}{b}\Big),0,0\right)$ 
is stable;
\item co-existence steady state 
$\displaystyle (E^*_{uC}, F^*_{uC}, E^*_{iC}, F^*_{iC})$ is unstable, with
$F^*_{uC}=\frac{\nu\beta_F}{\delta_F}E^*_{uC}$, $F^*_{iC}=\frac{\nu\beta_F}{\delta\delta_F}E^*_{iC}$ and
$$
E^*_{uC}:= \frac{1}{s_h+\delta-1}\left(\frac{\eta}{\delta}-(1-s_h)K\left(1-\frac{\delta\delta_F}{b\eta}\right)\right), \quad 
E^*_{iC}:=\frac{1}{s_h+\delta-1}\left(\delta K\left(1-\frac{\delta\delta_F}{b\eta}\right)-\frac{\eta}{\delta}\right);
$$
\item extinction $(0,0,0,0)$ is unstable.
\end{itemize}
\end{lemma}
Notice that the first assumption in \eqref{hyp:coeffW} boils down to consider that the birth rate is larger than the death rate and is generically satisfied for mosquitoes population. Since $s_h$ is expected to be close to $1$ (the case $s_h=1$ being the perfect cytoplasmic incompatibility case), the second inequality may be seen as a condition on $K$ to be large enough.

As above, we denote $u$ the release function of {\it Wolbachia}-infected mosquitos.
Assume that the system is initially at the {\it Wolbachia} free equilibrium,
we want to determine an optimal release function $u$ which brings the system as close as possible to the {\it Wolbachia} invasion equilibrium.
More precisely, let us consider $(E_u,F_u,E_i,F_i)$ solution to the following Cauchy problem:
\begin{equation}\label{pb:Wolbachia}
  \left\{
      \begin{aligned}
\frac{d}{dt} E_u &= \beta_E F_u \left(1-s_h \frac{F_i}{F_u+F_i}\right)\left(1-\frac{E_u+E_i}{K}\right) - \big(\tau_E+\delta_E\big) E_u,  \\
\frac{d}{dt} F_u &= \nu \beta_F E_u - \delta_{F} F_u, \\
\frac{d}{dt} E_i &= \eta \beta_E F_i\left(1-\frac{E_u+E_i}{K}\right) - \big(\tau_E+\delta_E\big) E_i,  \\
\frac{d}{dt} F_i &= \nu \beta_F E_i - \delta \delta_{F} F_i+u. \\
E_u(0) &= K\left(1-\frac{\delta_F}{b}\right),\quad F_u(0) = K\left(\frac{\nu\beta_F}{\delta_F}-\frac{\nu\beta_F}{b}\right), \quad E_i(0) = F_i(0) = 0.
\end{aligned}
\right.
\end{equation}
We introduce the following cost function
\begin{equation*}
J(u) = \frac 12 \left(E_u(T)^2 + F_u(T)^2 + \Big(K\left(1-\frac{\delta\delta_F}{b\eta}\right) - E_i(T)\Big)_+^2 + \Big(K\left(\frac{\nu\beta_F}{\delta\delta_F}-\frac{\nu\beta_F}{b\eta}\right) - F_i(T)\Big)_+^2\right),
\end{equation*}
with the standard notation for the positive part $X_+=\max \{X,0\}$ for $X\in \mathbb{R}$.
We investigate the following optimisation problem
\begin{equation}\label{pb_Wol}
\min_{u\in {\mathcal U}_{C,\overline{U}}} J(u), \qquad
{\mathcal U}_{C,\overline{U}} = \left\{0\leq u\leq \overline{U}, \quad \int_0^T u(t)dt \leq C\right\}.
\end{equation}
The following result give the existence of a solution. Its proof is postponed in the Appendix.
\begin{proposition}\label{prop:wolbachiaExistsOC}
Under the same assumptions as in Lemma \ref{lem:eqWol}, problem \eqref{pb_Wol} has a solution.
\end{proposition}

System \eqref{eq:S4} may be even more simplified by assuming a fast dynamics of the aquatic phase and a large fertility, in the spirit of \cite{Stru2016}. It leads to a simple differential equation on the proportion of infected female mosquitoes, for which the optimisation problem has been studied in detail in \cite{optim}. In particular, it has been proved for this simplified system that the optimal strategy uses the maximal amount of mosquitoes, in other words that $\int_0^T u^*(t)\, dt=C$ whenever $\overline{U}T>C$. It is likely that the same property holds true when considering the more realistic system \eqref{pb:Wolbachia} even it seems more tedious to show it, as can be observed on simulations in the next section.

\section{Numerical simulations}
\label{sec:num}

We will now give some solutions of the optimal control problems \eqref{pb_TIS} and \eqref{pb_Wol}.
For this purpose, we will use the opensource optimization routine \textsc{GEKKO} (cf \cite{gekko}).
It enables the computation of a local minimizer of the optimization problem
using orthogonal collocation on finite elements to implicitly solve the differential algebraic equations.

\subsection{Sterile insect technique}

In this section, we will give some illustrations of the optimal strategy given by 
the optimal control problem \eqref{pb_TIS}.
We will use the parameters values of Table \ref{tab:steril} coming from \cite[Table 1-3]{bossin}.
We recall that $\nu$ is assumed to be equal to $0.5$.

\begin{table}[H]
\centering
\begin{tabular}{|c|c|c|c|}
\hline
\textit{Parameter}&\textit{Name}&\textit{Value interval}&\textit{Chosen value}\\\hline
 $\beta_E$&Effective fecundity&7.46--14.85&10  \\\hline
 $\gamma$&Mating competitiveness of sterilizing males&0--1&1  \\\hline
 $\tau_E$&Hatching parameter&0.005--0.25&0.05  \\\hline
 $\delta_E$&Mosquitoes in Aquatic phase death rate&0.023 - 0.046&0.03\\\hline
 $\beta_F$&Growth of female&0.005--0.025&0.010  \\\hline
 $\delta_F$&Female death rate&0.033 - 0.046&0.04\\\hline
 $\delta_s$&Infected male  death rate&& 0.12 \\\hline
\end{tabular}
\caption{Value intervals of the parameters for system \eqref{pb:sterile}}
\label{tab:steril}
\end{table}

As in \cite{bossin}, in order to get results relevant for an island
of 74 ha with an estimated male population of about 69 $ha^{-1}$, the density of male 
is equal to $M^*=F^*=69\times 74=5106.$
If we assume that $M_s=0$, we get
\begin{equation*}
\left\{\begin{array}{l}
0 = \beta_E F^*\left(1-\frac{E^*}{K}\right) - (\tau_E+\delta_E) E^*,  \\
0 = \nu\beta_F E^* - \delta_F F^*.
\end{array}\right.
\end{equation*}
Thus, the value of $K$ is given by the expression 
\begin{equation*}
K=\frac{E^*}{1-\frac{(\nu_E+\delta_E)\delta_F}{\nu\beta_F\beta_E}}\approx 5172.2.
\end{equation*}

\begin{figure}[h!]
\begin{center}
\includegraphics[width=4.9cm]{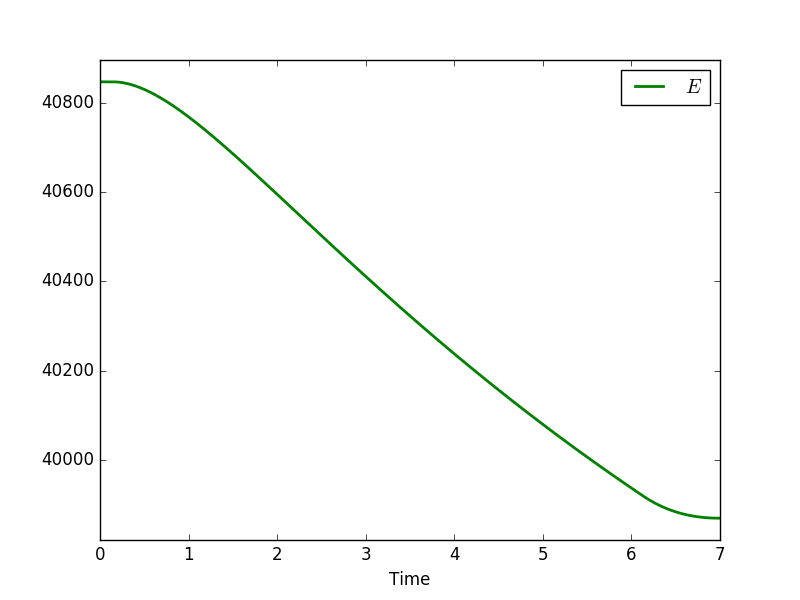}
\includegraphics[width=4.9cm]{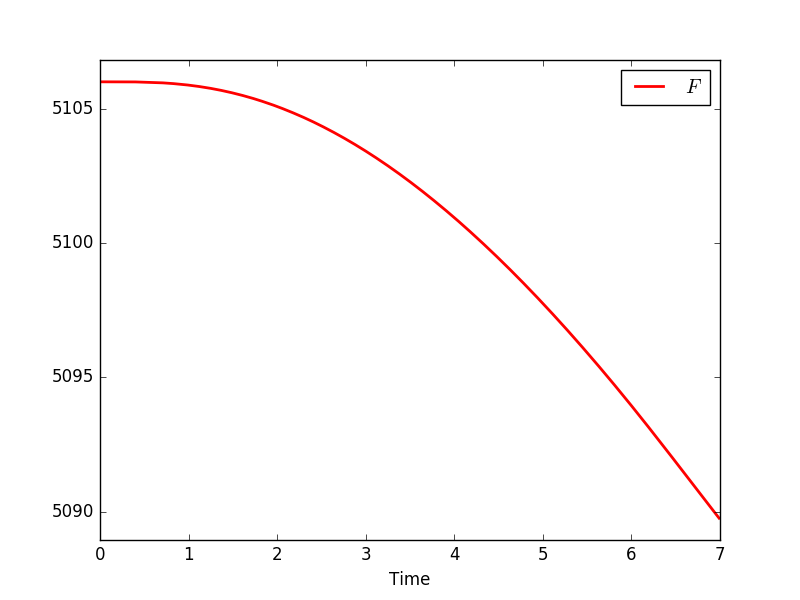}
\includegraphics[width=4.9cm]{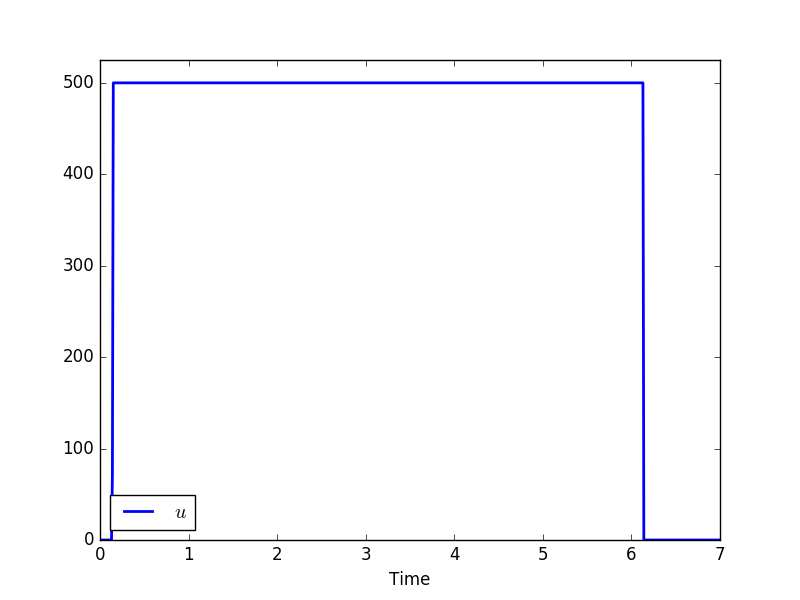}

\includegraphics[width=4.9cm]{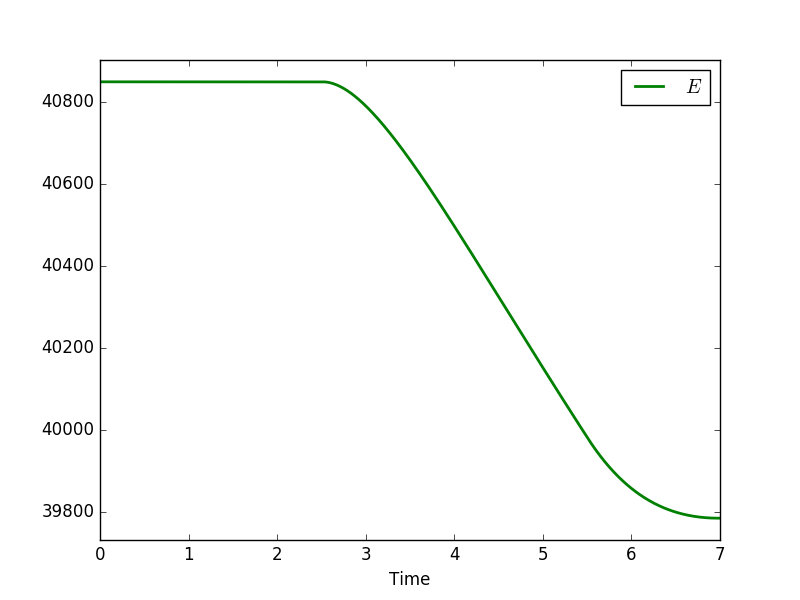}
\includegraphics[width=4.9cm]{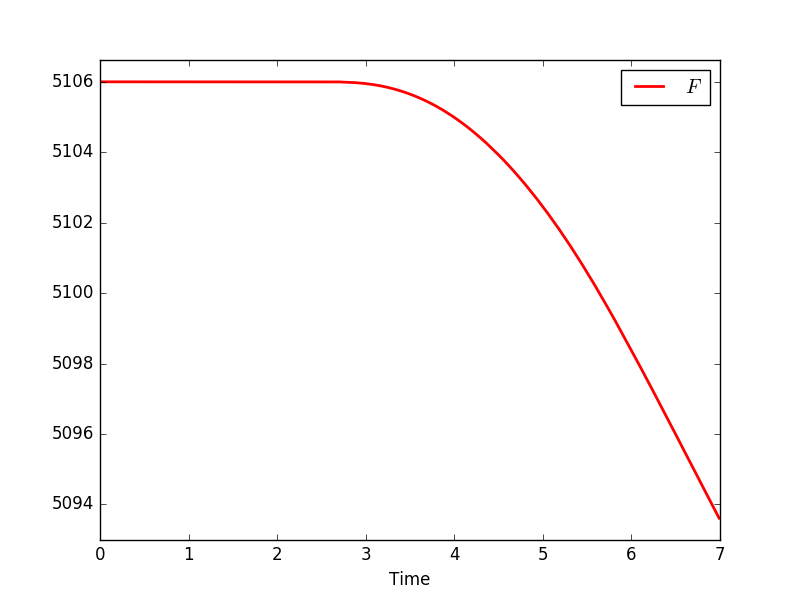}
\includegraphics[width=4.9cm]{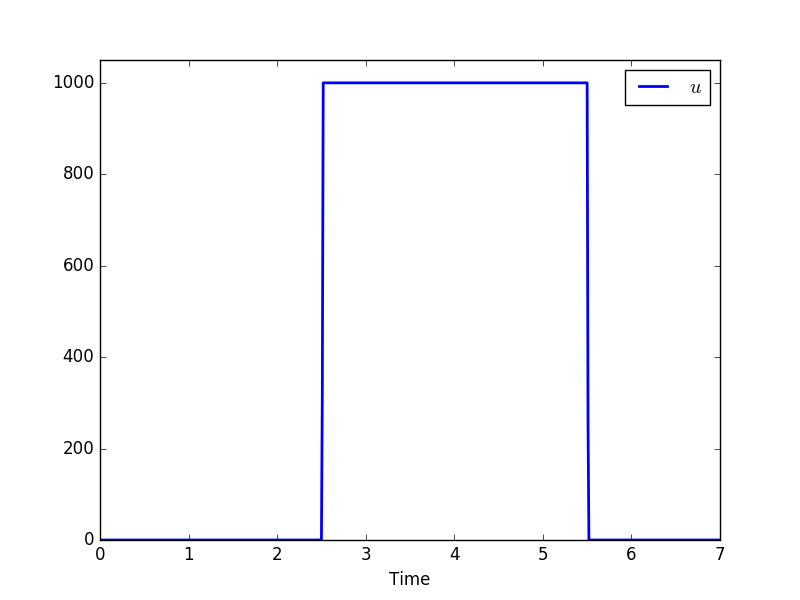}

\includegraphics[width=4.9cm]{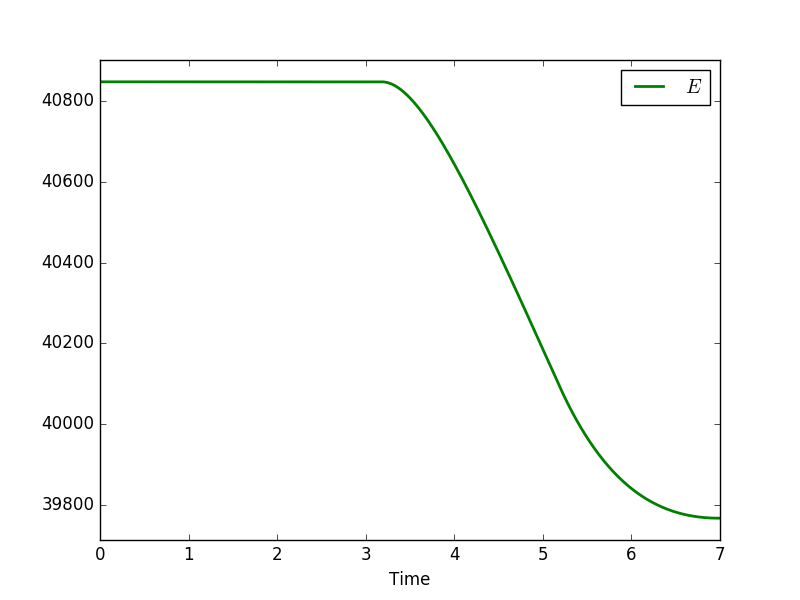}
\includegraphics[width=4.9cm]{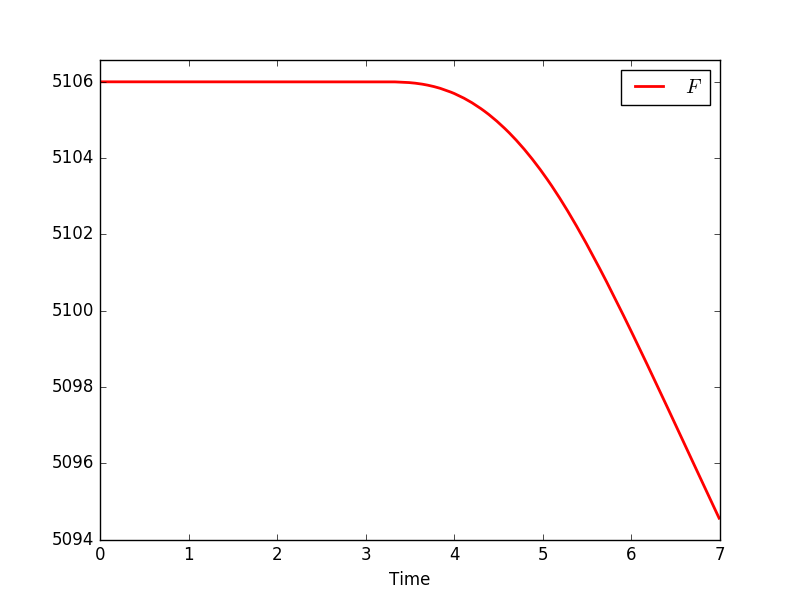}
\includegraphics[width=4.9cm]{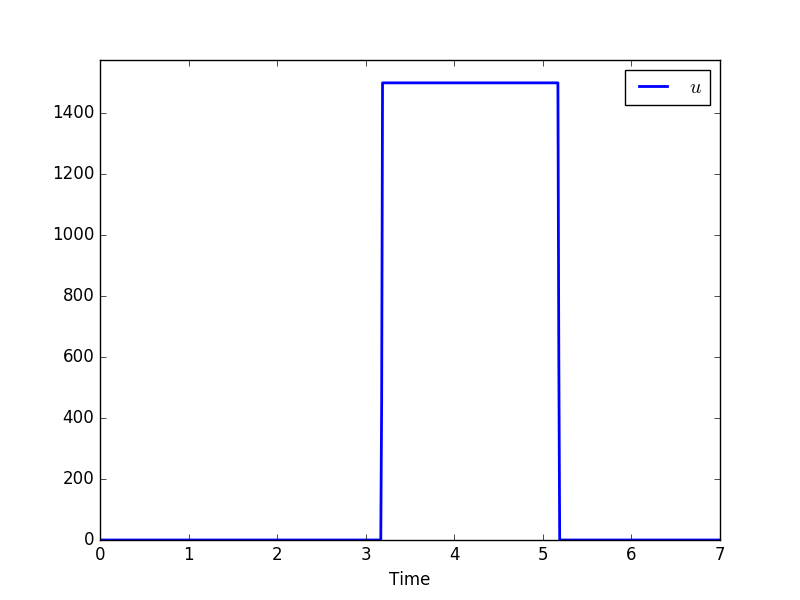}
\end{center}
\caption{Simulation of the sterile insect technique with the value of Table 
\ref{tab:steril}
for  $T=7$, $C=3000$, $\overline{U}=500$ (1st line), $1000$ (2nd line), $1500$ (3th line).}
\label{fig:steril}
\end{figure}

In Figure \ref{fig:steril}, we give some optimal strategies solution of the problem \eqref{pb_TIS}
on one week, i.e. $T=7$, a total quantity of sterile mosquitoes $C=3000$
and a maximum of instantaneous release equal to $\overline{U}=500$, $1000$ and $1500$.
As predicted by Proposition \ref{prop:existsterile}, the optimal control strategy uses the maximal amount of mosquitoes ($\int u=C$) and does not act at the end of the time interval.
Moreover, we can observe that 
\begin{itemize}
\item the optimal strategy is ``bang-bang'', \textit{i.e.} the optimal control
is either equal to zero or $\bar U$;
\item it seems preferable to mostly concentrate the releases at a particular time. 
\end{itemize}
It could be interesting in a future work to analyze this two phenomena.

\subsection{Population replacement}

The section is devoted to illustrate the optimal strategy produced by the optimal control problem \eqref{pb_Wol} with System \eqref{pb:Wolbachia} as constraint.
The parameters values for system \eqref{pb:Wolbachia} are given in Tables \ref{tab:steril} and \ref{tab:value wol}.
The expression of $b$ is given in Lemma \ref{lem:eqWol}.
The value of the cytoplasmic incompatibility parameter $s_h$ (corresponding to the fraction of eggs from uninfected females fertilize by infected males which will not hatch) comes from \cite{dutra2015lab}.
The fecundity reduction $\eta$ of infected females with respect to uninfected females
and the increase of mortality $\delta$ for infected mosquitoes
have been fixed following \cite{Hughes}.

\begin{table}[H]
\centering
\begin{tabular}{|c|c|c|c|}
\hline
\textit{Parameter}&\textit{Name}&\textit{Value interval}&\textit{Chosen value}\\\hline
 $s_h$&Probability ot cytoplasmic incompatibility&& 0.9951 \\\hline
    $\eta$& \begin{tabular}{l}Fecundity reduction of infected females\\ with respect to uninfected females
    \end{tabular}&0.85--1&0.95  \\\hline
     $\delta$&Increase of mortality for infected mosquitoes &1--1.7& 1.25 \\\hline
\end{tabular}
\caption{Value intervals of the parameters for System \eqref{pb:Wolbachia}}
\label{tab:value wol}
\end{table}
As for the sterile insect strategy, the initial density of mosquitoes will be equal to $F_u^*=5106$. 
We assume that the system \eqref{pb:Wolbachia} is initially at equilibrium $(E_u^*,F_u^*,0,0)$.
We can deduce the value of $K$ thanks to the expressions of the equilibrium $(E^*_{uE}, F^*_{uE}, E^*_{iE}, F^*_{iE})$ in Lemma \ref{lem:eqWol}.
The numerical results are displayed in Figure \ref{fig:wolbachia}, when we take a total amount of mosquitoes $C=10000$, and in Figure \ref{fig:wolbachia2} for a total amount of mosquitoes $C=1000$.
We first notice that, as in the case of the sterile insect technique, there exists a time after which the control function $u$ vanishes.

\begin{figure}[h]
\begin{center}
~\hfill\includegraphics[width=4.9cm]{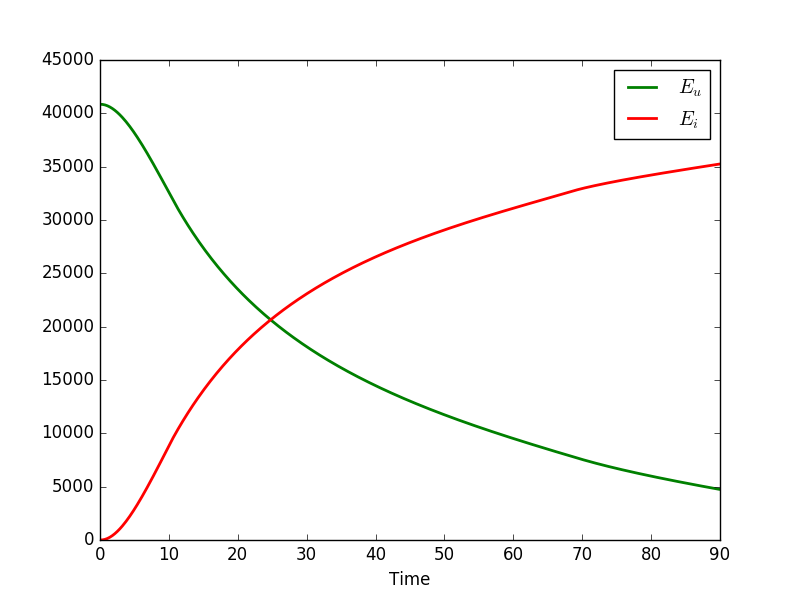}\hfill
\includegraphics[width=4.9cm]{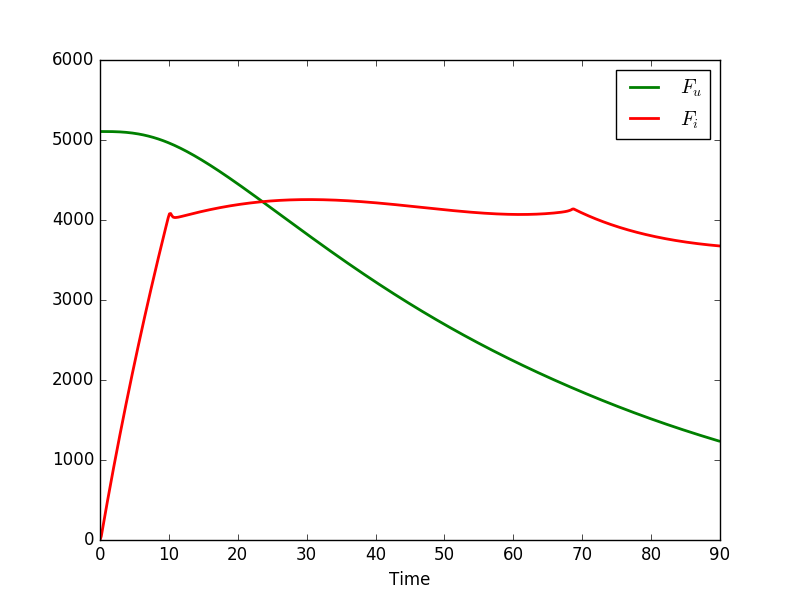}\hfill
\includegraphics[width=4.9cm]{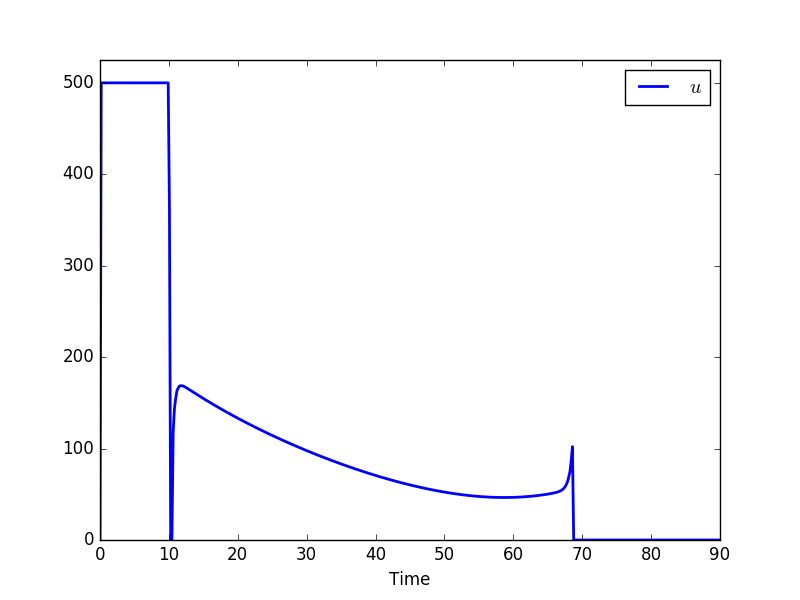}\hfill~

~\hfill\includegraphics[width=4.9cm]{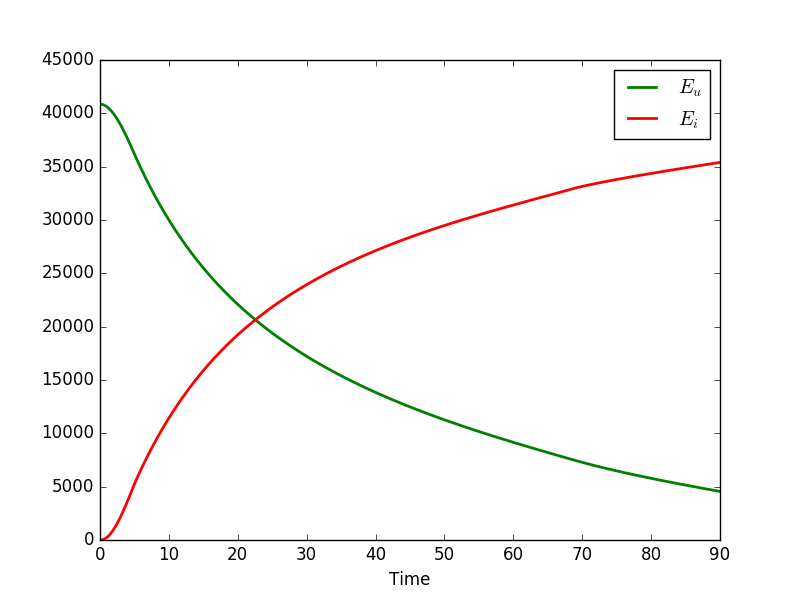}\hfill
\includegraphics[width=4.9cm]{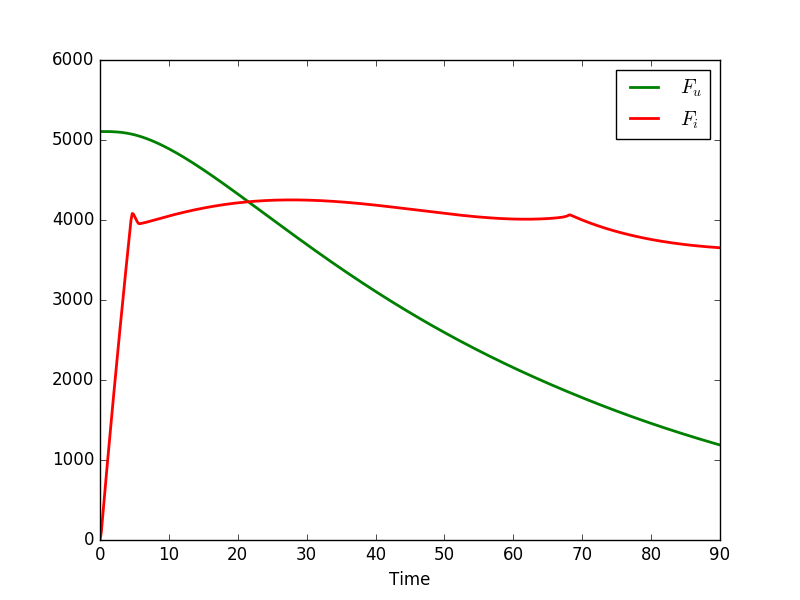}\hfill
\includegraphics[width=4.9cm]{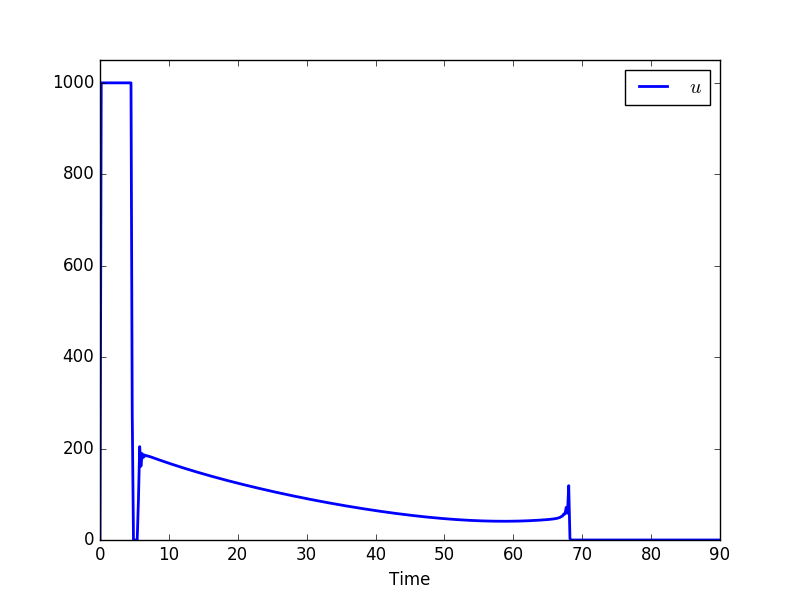}\hfill~

~\hfill\includegraphics[width=4.9cm]{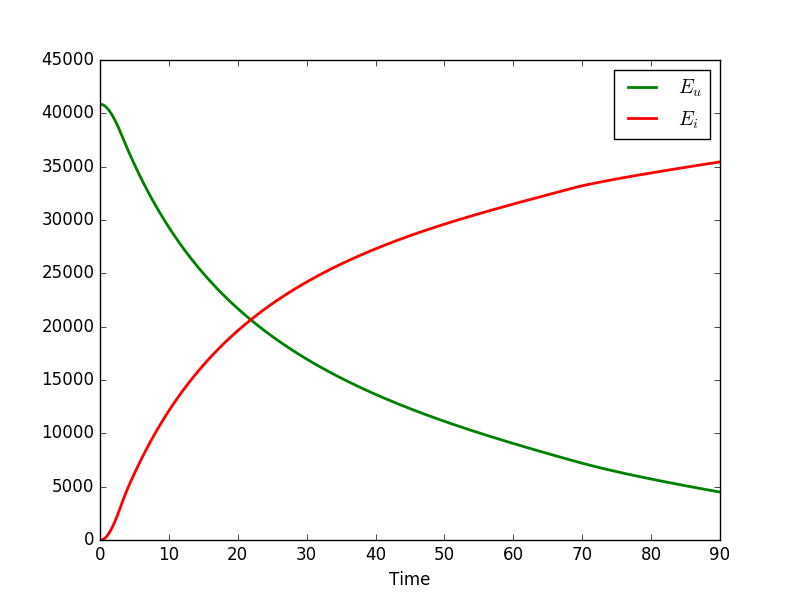}\hfill
\includegraphics[width=4.9cm]{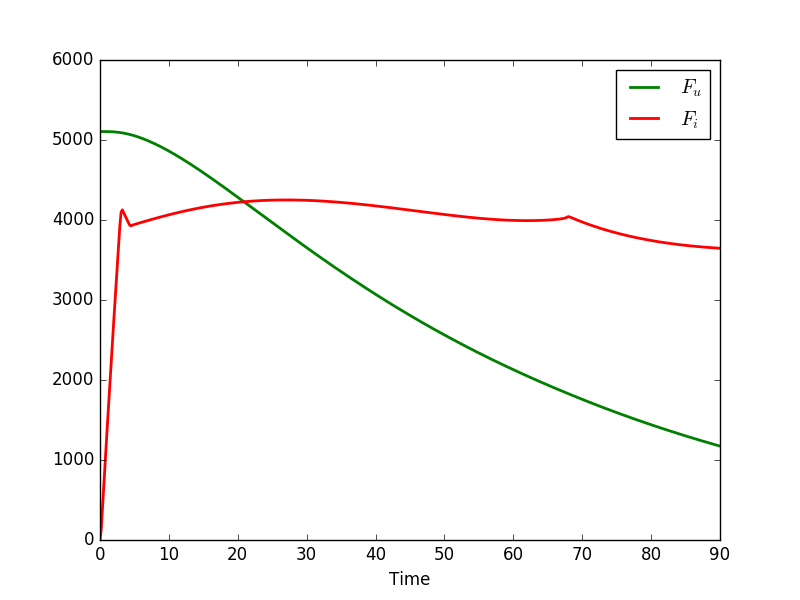}\hfill
\includegraphics[width=4.9cm]{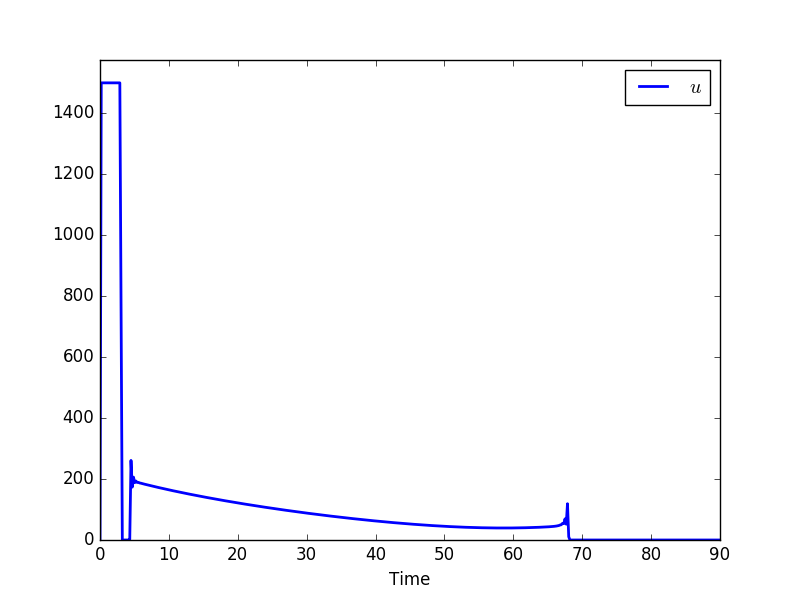}\hfill~
\end{center}
\caption{Simulation of the wolbachia technique with the values of Tables 
\ref{tab:steril} and \ref{tab:value wol}
for  $T=90$, $C=10000$, $\overline{U}=500$ (1st line), $1000$ (2nd line), 
$1500$ (3th line).}
\label{fig:wolbachia}
\end{figure}

\begin{figure}[h]
\begin{center}
~\hfill\includegraphics[width=4.9cm]{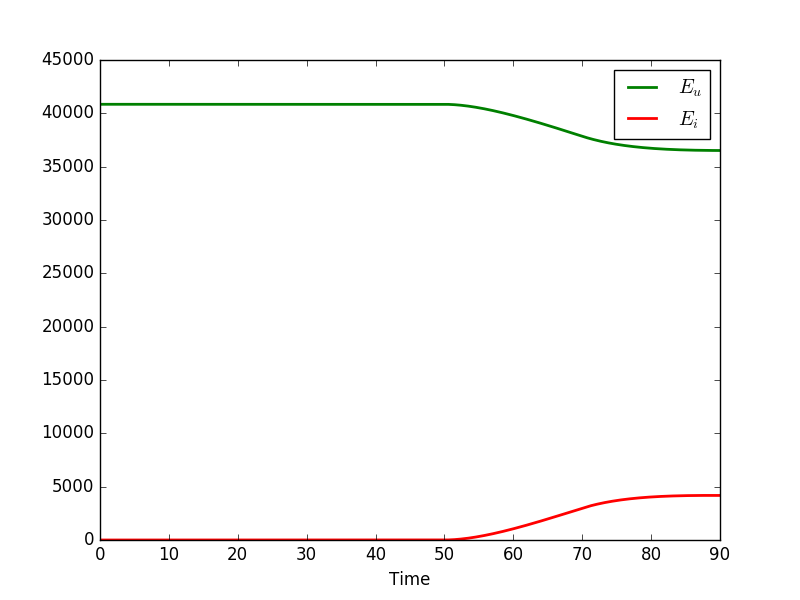}\hfill
\includegraphics[width=4.9cm]{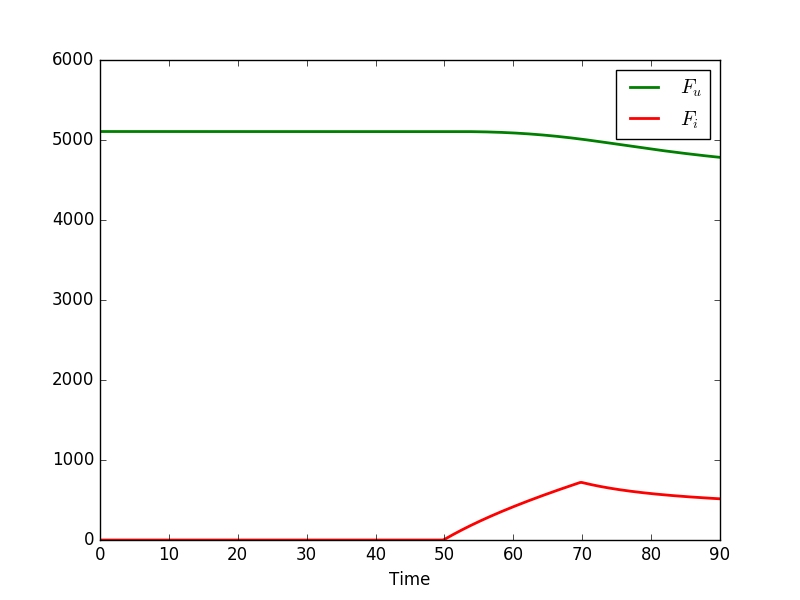}\hfill
\includegraphics[width=4.9cm]{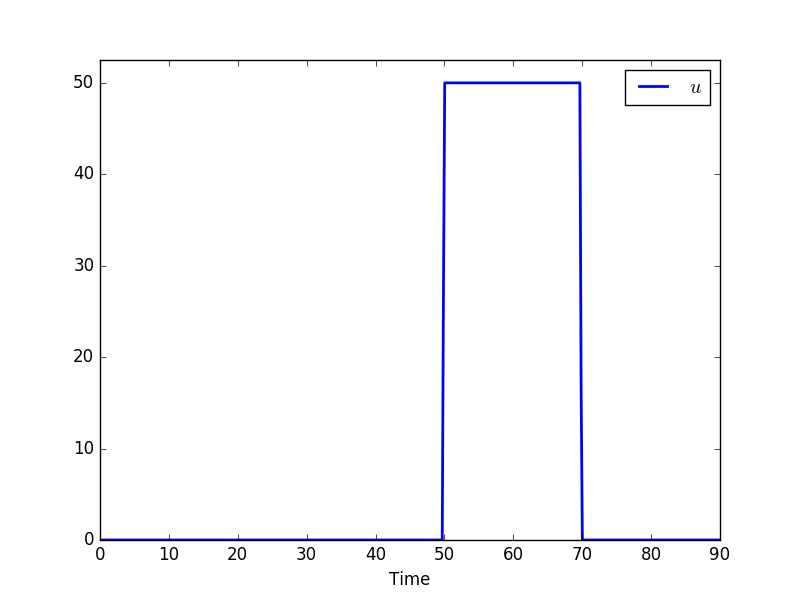}\hfill~

~\hfill\includegraphics[width=4.9cm]{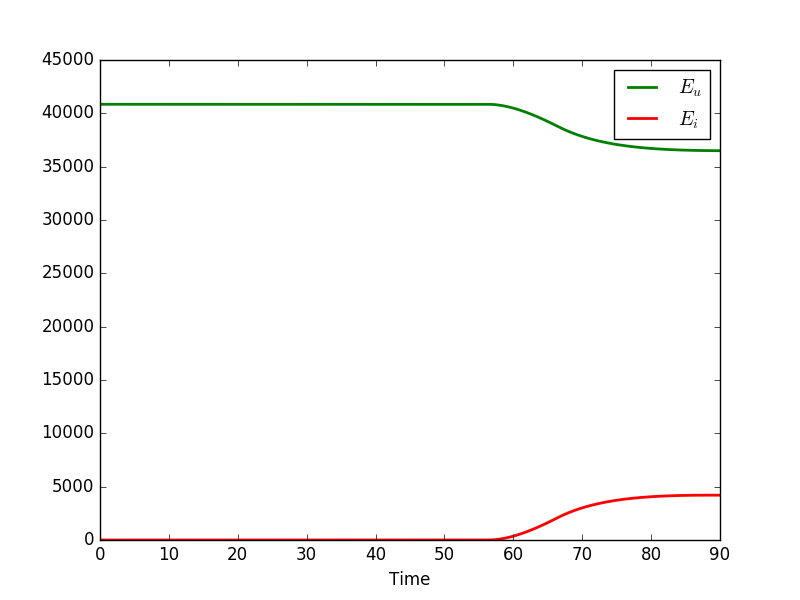}\hfill
\includegraphics[width=4.9cm]{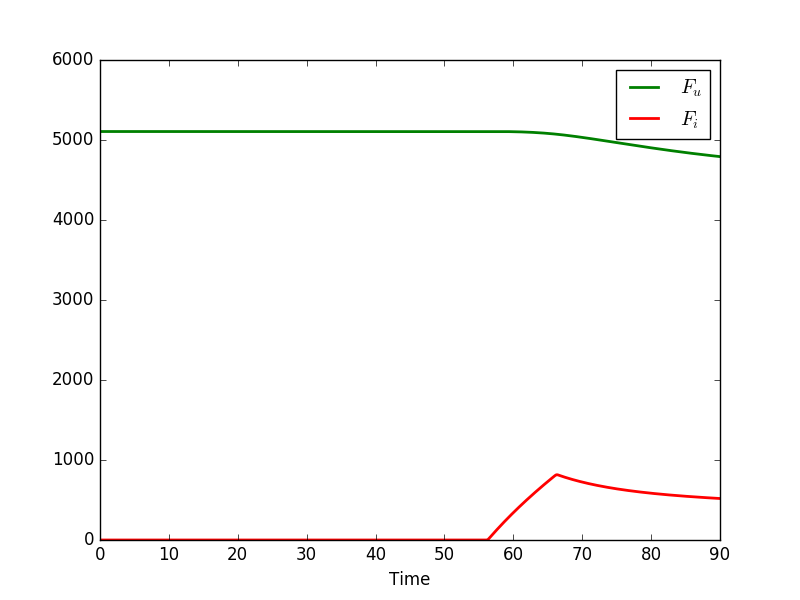}\hfill
\includegraphics[width=4.9cm]{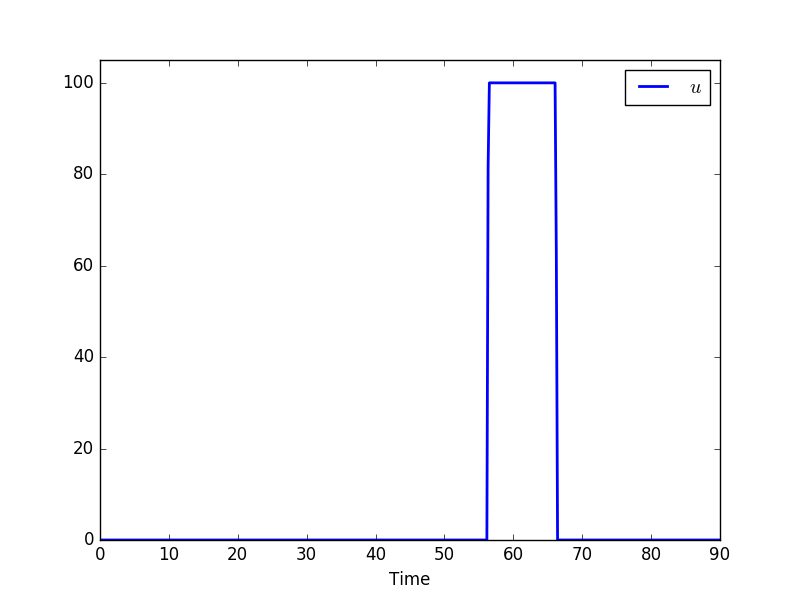}\hfill~

~\hfill\includegraphics[width=4.9cm]{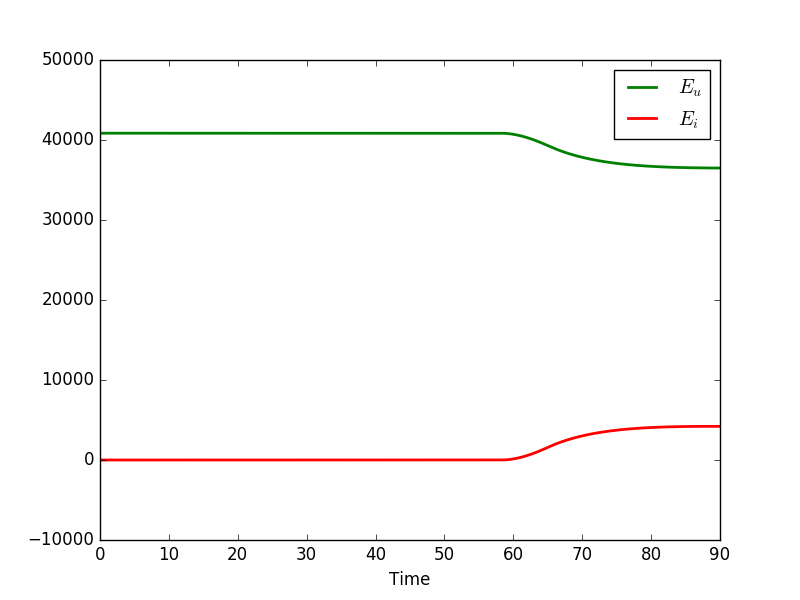}\hfill
\includegraphics[width=4.9cm]{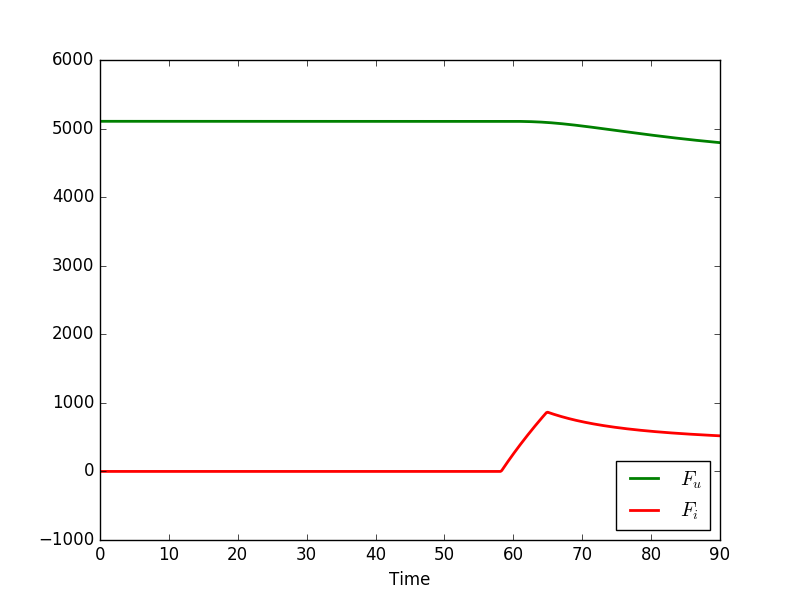}\hfill
\includegraphics[width=4.9cm]{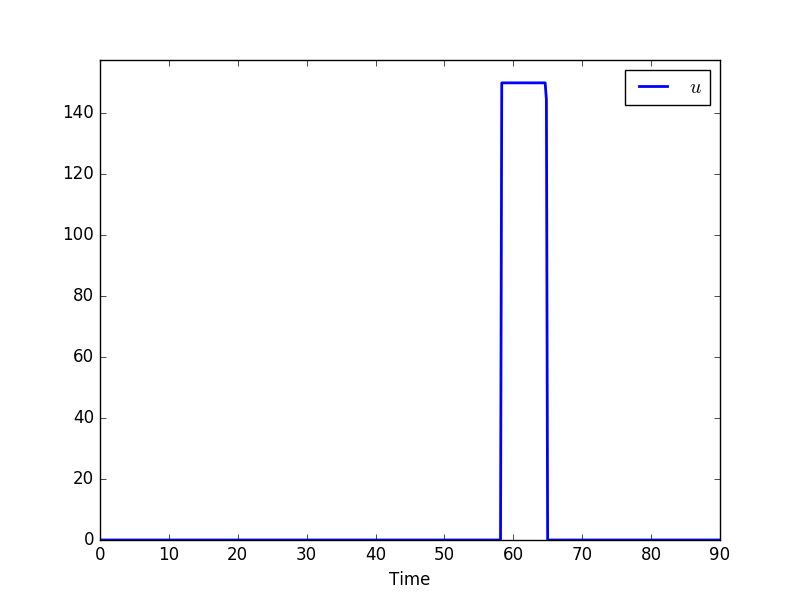}\hfill~
\end{center}
\caption{Simulation of the wolbachia technique with the values of Tables 
\ref{tab:steril} and \ref{tab:value wol}
for  $T=90$, $C=1000$, $\overline{U}=50$ (1st line), $100$ (2nd line), 
$150$ (3th line).}
\label{fig:wolbachia2}
\end{figure}

We can draw also some conclusions by comparing the Figures \ref{fig:wolbachia} and \ref{fig:wolbachia2}. Indeed, we observe that when the amount of mosquitoes is large enough, it is better to act at the beginning of the process. Whereas if the amount of mosquitoes is small, it seems better to make the releases later.
This observation should be related to the threshold phenomenon which has been observed in \cite{optim}, where the authors have approximated the model by a single equation on the proportion of infected adult mosquitoes.
For this very simple system, the authors have proved that the control is bang-bang and that there exists a threshold on the total amount of mosquitoes above which it is better to act at the end of the time interval and below which the action occurs at the beginning of the time interval.
For the model \eqref{pb:Wolbachia} with six equations, 
we remark, in Figures \ref{fig:wolbachia} and \ref{fig:wolbachia2}, that the optimal strategy is more complex.
However, the threshold phenomenon seems still to be true.

\section{Discussion and conclusion}

In this paper, we investigate the problem of optimizing a release protocol for a population replacement strategy and for the sterile insect technique applied to the control of \textit{Aedes} mosquitoes population. In our approach, we look for a control function $u$ minimizing the distance to the desired equilibrium (replacement or extinction of the wild population) at the final time of treatment. We show, in particular, the existence of such an optimal control and, after establishing some properties, we illustrate it thanks to numerical simulations. 

We discuss now some limitations of our models which will be addressed in a future work \cite{postdoc}. 
First, in both situations, we use a functional cost consisting in measuring a distance in the sense of mean least square to the objective equilibrium. 
Thanks to this choice, we are able to provide the temporal distribution of the release function in order to be as close as possible to the steady state aimed (extinction for the sterile insect technique, total invasion by \textit{Wolbachia} for the population replacement strategy).
This approach is totally justified when we have a given number of mosquitoes to release for a given duration and want to optimize the release protocol.
However, we might be interested by different approach; for example we might be interested by minimizing the number of mosquitoes to use, since the production of such mosquitoes may be costly, financially speaking.
Or we might be interested by reaching the steady state aimed. Indeed, with the functional costs used in this paper, we are not able to guarantee the success of the strategy.
To answer to this interesting question, we need first to determine the basin of attraction for the aimed steady state and then use a different functional cost for which the study performed in this paper should be adapted.

Second, for the mathematical modelling, we have done several assumptions in the aim to simplify the system in order to derive models that could be tractable for a mathematical study. Some of these assumptions should be weakened. In particular, as already mentioned, the pertinence of system \eqref{eq:S1} for low population is not clear. Since the population of mosquitoes is usually high, the use of such models is often justified. However, when we aim at eradicating this population by use of sterile insect technique, the behaviour of the system close to the extinction steady state is important. 
Therefore, a model which describes more carefully the dynamics close to the extinction steady state is necessary.
With this aim, a strategy consists in using a birth rate $\beta_E$ depending on the male density. 
For instance, in \cite{bossin}, a function $\beta_E$ depending exponentially of the male density has been considered, taking into account an Allee effect which guarantees the stability of the extinction steady state.

Moreover, it has been observed that the death rate for males may be higher than the one for females. Thus, assuming that the number of male mosquitoes is the same as the number of female one is a very strong assumption that should be weakened.

Nevertheless, the present work, and the rigorous mathematical results that we were able to prove in this simple simplified setting, should be a useful step towards the future understanding of more general and more realistic models.

\section*{Acknowledgments}
The authors were partially supported by the Project ``Analysis and simulation of optimal shapes - application to lifesciences'' of the Paris City Hall.

\section*{Conflict of interest}
The authors declare that they have no conflict of interest.

\appendix

\section{Proof of Proposition \ref{prop:sanseffetAllee}}
Thanks to Assumption \eqref{eq:cond prop}, there exists $\varepsilon>0$ small enough such that 
\begin{equation}\label{eq:alpha}
(\tau_{E}+\delta_E)\delta_F < \frac{1}{1+\varepsilon}\nu \beta_E \beta_F.
\end{equation}
Using that  
$M_s(t)=e^{-\delta_st}M_{s0}$, $M(t)\geqslant e^{-\delta_Mt}M_{0}$ for $t\geqslant0$
and \eqref{eq:cond prop}, 
we deduce that there exists $t^*>0$ such that 
$$
\gamma M_s(t)<\varepsilon M(t) \quad \mbox{ for all }t\geqslant t^*.
$$
Let us assume by contradiction that the extinction steady state is stable. Then, we place ourselves in a neighborhood of this equilibrium in which system \eqref{eq:sterile1} is monotonous.
We deduce from standard comparison principle for monotonous system that, for all $t\geqslant t^*$ 
\begin{equation}\label{eq:Ei Fi}
E_1(t)\leqslant E(t),\quad
F_1(t)\leqslant F(t),\quad M_1(t)\leqslant M(t),
\end{equation}
where, $(E_1,F_1,M_1)$ solves the following system, for $t\geqslant t^*$,
\begin{equation}\label{system1}
  \left\{
  \begin{aligned}
\frac{d}{dt} E_1 &= \beta_E F_1 \left(1-\frac{E_1}{K}\right)\frac{1}{1+\varepsilon} - (\tau_E+\delta_E) E_1,  \\
\frac{d}{dt} F_1 &= \nu \beta_F E_1 - \delta_F F_1, \\
\frac{d}{dt} M_1 &= (1-\nu) \beta_F E_1 - \delta_M M_1.
  \end{aligned}
  \right.
\end{equation}
complemented with initial data $(E_1,F_1,M_1)(t^*)=(E,F,M)(t^*)$.
We may study the stability of the extinction steady state for this later system. The Jacobian of this system in $0$ is given by
$$
\text{Jac}(0) = \begin{pmatrix}
-(\tau_E+\delta_E) & \frac{\beta_E}{1+\varepsilon} & 0  \\
\nu \beta_F & -\delta_F & 0  \\
(1-\nu) \beta_F & 0 & -\delta_M \\
\end{pmatrix}
$$
The characteristic polynomial for this matrix is given by
$$
P_0(X) = -(X+\delta_M)\left(
X^2+X(\delta_F+\delta_E+\tau_E) +\delta_F(\tau_E+\delta_E) -\frac{\nu\beta_E\beta_F}{1+\varepsilon}
\right).
$$
We have that $P_0(x)\to -\infty$ as $x\to +\infty$, and under assumption \eqref{eq:alpha}, $P_0(0)>0$, then $P_0$ admits a positive root. 
Hence, Jac$(0)$ admits a positive eigenvalue, and the extinction state for system \eqref{system1} is unstable.
From the comparison in \eqref{eq:Ei Fi}, we conclude the proof.

\section{Proof of Proposition \ref{prop:existsterile}}

For the analysis of the optimal control problem \eqref{pb_TIS}, it will be useful to notice that the solutions of System \eqref{pb:sterile} remain bounded.

\begin{lemma} \label{lem:1403}
Let $u\in {\mathcal U}_{C,\overline{U}} $ and $(E,F,M_s)$ be the solution of System \eqref{pb:sterile} associated to the control function choice $u$. For every $t\in [0,T]$, one has
$$
E_2^*e^{-(\tau_E+\delta_E)t}\leq E(t)< K \qquad\text{and}\qquad  F_2^*e^{-\delta_F t}\leq F(t)\leq 
K\left(\frac{\nu \beta_F}{\delta_F}-\frac{(\tau_E+\delta_E)}{\beta_E}e^{-\delta_F t}\right).
$$
\end{lemma}
\begin{proof}
Notice as a preliminary remark that a standard barrier argument ensures the positiveness of solutions to System \eqref{pb:sterile}. Let us show the right inequality on $E(\cdot)$. One has $E(0)<K$. Assume by contradiction the existence of $t_0\in (0,T]$ such that $E(t_0)=K$. Without loss of generality, we assume that $t_0$ is the first solution of the equation $E(t)=K$ on $(0,T]$. Since $\frac{dE}{dt}(t_0)=-(\tau_E+\delta_E)K<0$, we infer that $E(t)>K$ for $t<t_0$, close enough to $t_0$ whence a contradiction. 

The left-hand side inequality on the function $E(\cdot)$ follows directly from the observation that the right-hand side of the first equation of System \eqref{pb:sterile} is bounded by below by $-(\tau_E+\delta_E)E$ and a Gronwall argument.

Regarding now the inequalities on $F$, we claim that the left inequality follows from the positiveness of $E$. Moreover, by using that $E(\cdot)<K$ and the expression of $b$, we get
\begin{eqnarray*}
F(t) &=& e^{-\delta_F t}F(0)+\nu \beta_F\int_0^t e^{-\delta_F (t-s)}E(s)\, ds 
\leq K\left(\frac{\nu \beta_F}{\delta_F}-\frac{(\tau_E+\delta_E)}{\beta_E}e^{-\delta_F t}\right).
\end{eqnarray*}
\end{proof}

\paragraph{Existence of an optimal control. }Let us consider a minimizing sequence $(u_n)_{n\in \mathbb{N}}$ and denote by $(E_n,F_n,{M_s}_n)_{n\in \mathbb{N}}$ the corresponding solution to System \eqref{pb:sterile}. Noting that the class ${\mathcal U}_{C,\overline{U}} $ of admissible controls is compact for the $L^\infty$ weak-star topology, we infer the existence of $u^*\in {\mathcal U}_{C,\overline{U}}$ such that $(u_n)_{n\in \mathbb{N}}$ converges up to a subsequence to $u^*$  for the $L^\infty$ weak-star topology. Since
$$
{M_s}_n:\mathbb{R}_+\ni t\mapsto \int_0^t e^{-\delta_S(t-s))}u_n(s)\, ds,
$$
the sequence $({M_s}_n)_{n\in \mathbb{N}}$ converges in $H^1(0,T)$ up to a subsequence to $M_s^*$ given by 
$$
M_s^*:\mathbb{R}_+\ni t\mapsto \int_0^t e^{-\delta_S(t-s))}u^*(s)\, ds.
$$
According to Lemma \ref{lem:1403}, the triple $(E_n,F_n,{M_s}_n)$ is uniformly bounded on $[0,T]$. By using this boundedness property, one easily gets that $(\frac{dE_n}{dt})_{n\in \mathbb{N}}$ and $(\frac{dF_n}{dt})_{n\in \mathbb{N}}$ are bounded in $C^0([0,T])$ and therefore, $(E_n,F_n)_{n\in \mathbb{N}}$ is bounded in $W^{1,\infty}(0,T)$. According to the Ascoli theorem, the sequence $(E_n,F_n)_{n\in \mathbb{N}}$ converges to some $(E^*,F^*)\in W^{1,\infty}(0,T)$ in $C^0([0,T])$. As a consequence, according to \eqref{pb:sterile}, $(\frac{dE_n}{dt},\frac{dF_n}{dt},\frac{d{M_s}_n}{dt})_{n\in \mathbb{N}}$ is bounded in $L^2(0,T)$ and therefore, $(E_n,F_n,{M_s}_n)_{n\in \mathbb{N}}$ also converges (up to a subsequence) to  $(E^*,F^*,M_s^*)$ in $H^1(0,T)$. We then infer from all the considerations above that $(E^*,F^*,M_s^*)$ satisfies System \eqref{pb:sterile} and that $(J(u_n))_{n\in \mathbb{N}}$ converges, up to a subsequence, to $J(u^*)$. The existence follows.

\paragraph{First order optimality conditions.} 
Let $u^*$ be an optimal control for Problem \eqref{pb_TIS} and $(E^*,F^*,M_s^*)$ be the corresponding trajectories, solutions of \eqref{pb:sterile} for $u=u^*$. To write the first order optimality conditions, we will use the Pontryagin Maximum Principle (PMP) . To take into account the integral constraint on $u$, it is convenient to introduce a new state variable $y$ solving the o.d.e.
$$
y'(t)=u(t)\quad \text{on }[0,T]\qquad \text{and}\qquad y(0)=0
$$
in such a way that the constraint $\int_0^T u(t)\, dt\leq C$ rewrites as the terminal condition $y(T)\leq C$.

Let us introduce the function $f_E$ defined by
$$
f_E(E,F,M_s)= \beta_E F\left(1-\frac{E}{K}\right)\frac{F}{F+\gamma M_s}
$$
as well as the Hamiltonian of Problem \eqref{pb_TIS}, given by
\begin{eqnarray*}
\mathcal{H}((E,F,M_s,y),(p_1,p_2,p_3,\lambda),u)&=& p_1 \left(f_E(E,F,M_s) - (\tau_E+\delta_E) E\right)+p_2\left(\nu\beta_F E - \delta_F F\right)\\
&& +p_3\left(u - \delta_S M_s\right)+\lambda u.
\end{eqnarray*}
According to the Maximum Principle (see, e.g. \cite{Lee-Markus}), there exist an absolutely continuous mapping $p:[0,T]\to \mathbb{R}^3$ called adjoint vector such that the so-called extremal $((E^*,F^*,M_s^*,y^*),(p_1^*,p_2^*,p_3^*,\lambda^*),u^*)$ satisfies a.e. in [0,T]:
\begin{itemize}
\item {\it Adjoint equations:} 
\begin{equation}\label{metz-strasb0711}
-\frac{d}{dt}\begin{pmatrix}p_1^*\\ p_2^*\\ p_3^*\end{pmatrix}=\begin{pmatrix}
\frac{\partial f_E}{\partial E}(E^*,F^*,M_s^*)-(\tau_E+\delta_E) & \nu\beta_F & 0\\
\frac{\partial f_E}{\partial F}(E^*,F^*,M_s^*) & -\delta_F & 0\\
\frac{\partial f_E}{\partial M_s}(E^*,F^*,M_s^*) & 0 & -\delta_s
\end{pmatrix}\begin{pmatrix}p_1^*\\ p_2^*\\ p_3^*\end{pmatrix}
\end{equation}
and in addition, ${\lambda^*}'=0$ which implies that $\lambda^*$ is a constant (still denoted $\lambda^*$ with a slight abuse of notation). 
\item {\it Minimality condition:} 
$$
\text{for a.e. }t\in [0,T], \ u^*(t)\text{ solves the problem }\min_{0\leq v\leq \bar U}(p_3^*+\lambda^*) v
$$
and therefore, one has 
\begin{equation}\label{metz-strasb0712}
p_3^*+\lambda^*\geq 0 \text{ on }\{u^*=0\}\qquad \text{and}\qquad p_3^*+\lambda^*\leq 0 \text{ on }\{u^*=\bar U\}.
\end{equation}
\item {\it Transversality conditions:} we impose the terminal conditions 
\begin{equation}\label{metz-strasb0713}
p_1^*(T)=E^*(T), \quad p_2^*(T)=F^*(T), \quad p_3^*(T)=0, \quad \lambda^*(T)=\xi
\end{equation}
on the adjoint state, where $\xi\in \mathbb{R}_+$ satisfies moreover the complementary condition $\xi (y(T)-C)=0$.
\end{itemize}
We infer from \eqref{metz-strasb0713} that $\lambda^*\geq 0$.

\paragraph{The total number of mosquitoes is used.}

Let us start with a preliminary lemma, whose proof is postponed at the end of this section.

\begin{lemma}\label{strasb:1720}
Let us assume that $E(0)<K$. Then, the solution $(E,F,M_s)$ of System \eqref{pb:sterile} satisfies
$$
\frac{\partial f_E}{\partial E}(E,F,M_s) <  0, \quad \frac{\partial f_E}{\partial F}(E,F,M_s) > 0\quad \text{and}\quad \frac{\partial f_E}{\partial M_s} (E,F,M_s)<  0.
$$
\end{lemma}

Let us argue by contradiction, considering $u^*$ a solution of Problem \eqref{pb_TIS} and $(E^*,F^*,M_s^*)$ the associated trajectory. If $\int_0^T u^*(t)\, dt<C$, then one has necessarily $\xi=0$ or equivalently $\lambda^*=0$. We will reach a contradiction by showing that one has $p_3^*<0$ on $(0,T)$. Indeed, if $p_3^*<0$ on $(0,T)$, then one has necessarily $u^*=\overline{U}$ on $(0,T)$ according to \eqref{metz-strasb0712} and since $\lambda^*=0$. But $\overline{U}$ is not feasible according to condition \eqref{cond:barUT}, yielding a contradiction.

Let us show that $p_3^*<0$ on $(0,T)$. To this aim, we introduce
$$
A:t\mapsto \frac{\partial f_E}{\partial E}(E^*(t),F^*(t),M_s^*(t))-(\tau_E+\delta_E)\quad \text{and}\quad B=\frac{\partial f_E}{\partial F}(E^*(t),F^*(t),M_s^*(t)).
$$
Then, the two first equations of the adjoint system \eqref{metz-strasb0711} read
$$
\left\{\begin{array}{ll}
(p_1^*(t)e^{\int^t A})' = -\nu\beta_F p_2^*(t) e^{\int^t A} & \text{on }[0,T],\\
(p_2^*(t)e^{-\delta_F t})' = -B(t) p_1^*(t) e^{-\delta_F t} & \text{on }[0,T].
\end{array}
\right.
$$
Let $v:t\mapsto p_1^*(t)e^{\int^t A}$. The last system becomes
$$
\left\{\begin{array}{ll}
v' = -\nu\beta_F p_2^*(t) e^{\int^t A} & \text{on }[0,T],\\
(p_2^*(t)e^{-\delta_F t})' = -B(t) v e^{-\int^t A}e^{-\delta_F t} & \text{on }[0,T].
\end{array}
\right.
$$
Therefore, $v$ solves the ODE
$$
(e^{-\int^t(\delta_F+A)}v'(t))'=\nu\beta_FBe^{-\int^t(\delta_F+A)}v(t).
$$
Notice that
$$
v(T)=p_1^*(T)e^{\int^TA}>0\quad \text{and}\quad v'(T)=((p_1^*)'(T)+A(T)p_1(T))e^{\int^TA}=-\nu\beta_Fp_2^*(T)e^{\int^TA}<0.
$$
Introduce the change of variable $s=\int_0^t e^{\int^z(\delta_F+A)}\, dz$, $\tilde T=\int_0^T e^{-\int^z(\delta_F+A)}\, dz$ and the function $w$ defined on $[0,\tilde T]$ by $w(s)=v(t)$. Then, the function $w$ satisfies the Cauchy system
$$
\left\{\begin{array}{ll}
w''(s) = \nu\beta_F B e^{-2\int^t(\delta_F+A)}w(s) & s\in [0,\tilde T],\\
w(\tilde T)=v(T)>0 ,& \\
w'(\tilde T)=v'(T)e^{-\int^T(\delta_F+A)}<0 ,& 
\end{array}
\right.
$$
where $t$ has to be understood as a function of $s$ in this system. We then infer that $w''(T)>0$ and therefore $w$ is convex in a neighborhood of $\tilde T$. Since it is also positive and decreasing according to the terminal conditions, it follows that $w$ cannot vanish on $[0,\tilde T]$. We successively infer that $v$ is positive on $[0,T]$ and so is $p_1^*$.  

Recall that $p_3^*$ satisfies the equation
$$
-(p_3^*)'=\frac{\partial f_E}{\partial M_s}(E^*,F^*,M_s^*)p_1^*-\delta_sp_3^*
$$
and therefore 
$$
(p_3^*e^{-\delta_s t})'=-e^{-\delta_s t}\frac{\partial f_E}{\partial M_s}(E^*,F^*,M_s^*)p_1^*>0
$$
according to the reasoning above and Lemma \ref{strasb:1720}. It follows that $t\mapsto p_3^*(t)e^{-\delta_s t}$ increases on $[0,T]$ and vanishes at $T$ only. Thus, $p_3^*<0$ on $[0,T)$ and we are done.

\paragraph{Structure of the control.}
We have shown that  $\lambda^*\neq0$ and therefore, $\lambda^*<0$. According to the first order optimality conditions (and \eqref{metz-strasb0712} in particular), since $p_3^*(T)=0$ and $p_3^*$ is continuous, we infer that $u^*=0$ in a neighborhood of $T$.

\begin{proof}[Proof of Lemma \ref{strasb:1720}]
Using Lemma \ref{lem:1403}, $E(t)<K$ for all $t\in[0,T]$. 
After some computations, we thus obtain
\begin{equation*}
\left\{\begin{array}{l}
\frac{\partial f_E}{\partial E}(E,F,M_s)=\frac{-\beta_EF^2}{K(F+\gamma M_s)}<0,\\\noalign{\smallskip}
\frac{\partial f_E}{\partial F}(E,F,M_s)=\left(1-\frac{E}{K}\right)\frac{\beta_EF^2+2\gamma\beta_EFM_s}{(F+\gamma M_s)^2}>0,\\\noalign{\smallskip}
\frac{\partial f_E}{\partial M_s}(E,F,M_s)=\frac{-\gamma \beta_EF^2\left(1-\frac{E}{K}\right)}{(F+\gamma M_s)^2}<0.
\end{array}\right.
\end{equation*}
\end{proof}

\section{Proof of Proposition \ref{prop:wolbachiaExistsOC}}
This proof is very similar to the one of Proposition \ref{prop:existsterile}. For the sake of completeness but to avoid redundancies, we only provide a sketch of proof. Let us consider a minimizing sequence $(u_n)_{n\in \mathbb{N}}$ and denote by $(E_u^n,F_u^n,{E_i}^n,F_i^n)_{n\in \mathbb{N}}$ the corresponding solution to System \eqref{pb:Wolbachia}.

\begin{itemize}
\item Since $(u_n)_{n\in \mathbb{N}}$ is uniformly bounded in $L^\infty(0,T)$, it converges to some element $u^*\in {\mathcal U}_{C,\overline{U}}$.
\item {\it The 4-tuple $(E_u^n,F_u^n,{E_i}^n,F_i^n)_{n\in \mathbb{N}}$ is bounded in $H^1(0,T)$.} First, observe that a standard barrier argument shows that each element of this tuple is positive. Moreover, given $n\in \mathbb{N}$, one has $E_i^n(t)+E_u^n(t)<K$. Indeed, one has $E_i^n(0)+E_u^n(0)\leq K$. Assuming by contradiction that $\max_{t\in [0,T]}E_i^n(t)+E_u^n(t)> K$, let $t_0$ be the first time in $(0,T)$ such that $E_i^n(t_0)+E_u^n(t_0)=K$. Then, one computes
$$
\frac{d}{dt} E_u^n(t_0) \leq  - \big(\tau_E+\delta_E\big) K<0\quad \text{and}\quad \frac{d}{dt} E_i^n(t_0) \leq - \big(\tau_E+\delta_E\big)K<0 ,
$$
yielding to a contradiction. It follows that $(E_u^n)_{n\in \mathbb{N}}$ and $(E_i^n)_{n\in \mathbb{N}}$ are bounded in $C^0([0,T])$. Since
$$
F_u^n(t)=K\left(\frac{\nu\beta_F}{\delta_F}-\frac{\nu\beta_F}{b}\right)e^{-\delta_Ft}+\nu\beta_F\int_0^t e^{-\delta_F(t-s)}E_u^n(s)\,
$$
and
$$
F_i^n(t)=\int_0^t e^{-\delta \delta_F(t-s)}(\nu\beta_FE_i^n(s)+u(s))\, ds,
$$
it follows that $(F_u^n)_{n\in \mathbb{N}}$ and $(F_i^n)_{n\in \mathbb{N}}$ are also bounded in $C^0([0,T])$.
Since $\frac{d}{dt} F_u^n \geq - \delta_{F} F_u^n$, a Gronwall inequality yields $ F_u^n(t)\geq  K\left(\frac{\nu\beta_F}{\delta_F}-\frac{\nu\beta_F}{b}\right)e^{-\delta_F t}$.
Therefore, $F_u^n+F_i^n$ cannot vanish on $[0,T]$ and we finally get the expected conclusion. 
\item By boundedness of $(E_u^n,F_u^n,{E_i}^n,F_i^n)_{n\in \mathbb{N}}$ in $(H^1(0,T))^4$, there exists $(E_u^*,F_u^*,E_i^*,F_i^*)\in (H^1(0,T))^4$ such that $(E_u^n,F_u^n,{E_i}^n,F_i^n)_{n\in \mathbb{N}}$ converges up to a subsequence to  $(E_u^*,F_u^*,E_i^*,F_i^*)\in (H^1(0,T))^4$, weakly in $H^1(0,T)$ and strongly in $L^2(0,T)$. Standard variational arguments show not only that $(E_u^*,F_u^*,E_i^*,F_i^*)$ satisfies \eqref{pb:Wolbachia} associated to the control function $u^*$, but also that $(J(u_n))_{n\in \mathbb{N}}$ converges to $J(u^*)$. The existence follows.
\end{itemize}


\begin{thebibliography}{99}


\bibitem{postdoc}
  \newblock L. Almeida, M. Duprez, Y. Privat and N. Vauchelet,
  \newblock Optimal release strategy for the sterile mosquitoes technique,
  \newblock work in progress.

\bibitem{optim}
  \newblock L. Almeida, Y. Privat, M. Strugarek and N. Vauchelet,
  \newblock Optimal releases for population replacement strategies, application to {\itshape Wolbachia},
  \newblock preprint, Hal-01807624 (2018).

\bibitem{Anguelov}
  \newblock R. Anguelov, Y. Dumont, and J. Lubuma,
  \newblock Mathematical modeling of sterile insect technology for control of anopheles mosquito,
  \newblock \emph{Comput. Math. Appl.}, \textbf{64} (2012), 374--389.

\bibitem{Barclay} 
    \newblock H.J. Barclay and M. Mackuer, 
    \newblock The sterile insect release method for pest control: a density dependent model, 
    \newblock \emph{Environ. Entomol.} \textbf{9} (1980), 810--817.

\bibitem{gekko}
  \newblock L. Beal, D. Hill, R. Martin and J. Hedengren,
  \newblock GEKKO Optimization Suite,
  \newblock \emph{Processes, Multidisciplinary Digital Publishing Institute}, \textbf{6} (8) (2018), 106.

\bibitem{bliman2}
  \newblock P.-A. Bliman,
  \newblock Feedback Control Principles for Biological Control of Dengue Vectors,
  \newblock preprint.

\bibitem{bliman1}
  \newblock P.-A. Bliman, M. S. Aronna, F. C. Coelho and M. A. Da Silva,
  \newblock Ensuring successful introduction of Wolbachia in natural populations of Aedes aegypti by means of feedback control,
  \newblock \emph{Journal of mathematical biology}, \textbf{76} (5) (2018), 1269--1300.

\bibitem{bliman3}
  \newblock P.-A. Bliman, D. Cardona-Salgado, Y. Dumont, O. Vasilieva,
  \newblock Implementation of Control Strategies for Sterile Insect Techniques,
  \newblock arXiv:1812.01277

  
\bibitem{bossin}
  \newblock H. Bossin, Y. Dumont and M. Strugarek,
  \newblock Using sterilizing males to reduce or eliminate Aedes populations:
  insights from a mathematical model,
  \newblock \emph{Applied Mathematical Modelling} (2019) \textbf{68}, 443--470.

\bibitem{Bourtzis}
  \newblock K. Bourtzis,
  \newblock \textit{Wolbachia}-based  technologies  for  insect  pest  population  control,
  \newblock In: Aksoy S. (eds) \emph{Transgenesis and the Management of Vector-Borne Disease. Advances in Experimental Medicine and Biology}, (2008) Vol. 627, Springer, New York, NY.

\bibitem{Cai}
    \newblock L. Cai, S. Ai, and J. Li, 
    \newblock Dynamics of mosquitoes populations with different strategies for releasing sterile mosquitoes,
    \newblock \emph{SIAM J. Appl. Math.} \textbf{74} (2014), pp. 1786--1809. 

\bibitem{colombien}
  \newblock D. E. Campo-Duarte,  O. Vasilieva,  D. Cardona-Salgado,  M. Svinin,
  \newblock Optimal control approach for establishing wMelPop Wolbachia infection among wild Aedes aegypti populations,
  \newblock \emph{Journal of mathematical biology}, \textbf{76} (2018), 1907--1950.
  
\bibitem{Dumont1}
  \newblock C. Dufourd, Y. Dumont,
  \newblock Impact of environmental factors on mosquito dispersal in th
  e prospect of sterile insect technique control,
  \newblock \emph{Comput. Math. Appl.}, \textbf{66} (2013), 1695--1715.

\bibitem{Dumont2}
  \newblock Y. Dumont, J.  M.  Tchuenche,
  \newblock Mathematical studies on the sterile insect technique for the chikungunya disease and Aedes albopictus,
  \newblock \emph{Journal of Mathematical Biology}, \textbf{65} (2012), 809--855.
 
\bibitem{dutra2015lab} 
    \newblock H. Dutra, L. dos Santos, E. Caragata, J. Silva, D. Villela, R. Maciel-de-Freitas and L. Moreira, 
    \newblock From lab to field: the influence of urban landscapes on the invasive potential of Wolbachia in Brazilian Aedes aegypti mosquitoes, 
    \newblock \emph{PLoS neglected tropical diseases} \textbf{9} (4) (2015), e0003689.

\bibitem{SIT}
  \newblock V.  A.  Dyck,  J.  Hendrichs,  and  A.  S.  Robinson,
  \newblock \emph{The  Sterile  Insect  Technique,  Principles  and Practice in Area-Wide Integrated Pest Management},
  \newblock Springer, Dordrecht, 2006

\bibitem{Farkas}
  \newblock J. Z. Farkas, P. Hinow,
  \newblock Structured and unstructured continuous models for wolbachia infections,
  \newblock \emph{Bulletin of Mathematical Biology}, \textbf{72} (2010), 2067--2088.

\bibitem{Fenton}
  \newblock A. Fenton, K. N. Johnson, J. C. Brownlie, G. D. D. Hurst,
  \newblock Solving the Wolbachia paradox: modeling the tripartite interaction between host, Wolbachia, and a natural enemy,
  \newblock \emph{The American Naturalist}, \textbf{178} (2011), 333--342.
  
 \bibitem{Huang}
  \newblock M.  Huang,  X.  Song,  and  J.  Li,
  \newblock Modelling and analysis of impulsive releases of sterile mosquitoes,
  \newblock \emph{Journal of Biological Dynamics}, \textbf{11} (2017), 147--171.

\bibitem{Hughes}
  \newblock H. Hughes, N. F. Britton,
  \newblock Modeling the Use of Wolbachia to Control Dengue Fever Transmission,
  \newblock \emph{Bull. Math. Biol.}, \textbf{75} (2013), 796--818.
  

 \bibitem{Lee-Markus}
  \newblock  E.B.  Lee  and  L.  Markus,
  \newblock Foundations  of  optimal  control  theory,
  \newblock \emph{SIAM  series  in  applied mathematics}. Wiley, 1967.

\bibitem{LiYuan}
  \newblock J. Li, Z. Yuan,
  \newblock Modelling releases of sterile mosquitoes with different strategies,
  \newblock \emph{Journal of Biological Dynamics}, \textbf{9} (2015), 1--14.

\bibitem{Sallet}
  \newblock G. Sallet, M. A. H. B. da Silva,
  \newblock Monotone dynamical systems and some models of Wolbachia in aedes aegypti populations,
  \newblock \emph{ARIMA}, \textbf{20} (2015), 145--176.
  
\bibitem{Schraiber}
  \newblock G.  Schraiber,  A.  N.  Kaczmarczyk,  R.  Kwok,  M.  Park,  R.  Silverstein,  F.  U. Rutaganira,  T. Aggarwal,  M. A.  Schwemmer,  C. L.  Hom,  R. K. Grosberg, S. J. Schreiber,
  \newblock Constraints on the use of lifespan-shortening wolbachia to control dengue fever,
  \newblock \emph{Journal of Theoretical Biology}, \textbf{297} (2012), 26--32.
  
\bibitem{Sinkins}
  \newblock S.  P.  Sinkins,
  \newblock Wolbachia  and  cytoplasmic  incompatibility  in  mosquitoes,
  \newblock \emph{Insect  Biochemistry and Molecular Biology}, \textbf{34} (2004), 723--729.
  
\bibitem{PhD}
    \newblock M. Strugarek,
    \newblock  \emph{Mod\'elisation math\'ematique de dynamiques de populations, applications \`a la lutte anti-vectorielle contre Aedes spp. (Diptera:Culicidae)},
    \newblock  Ph.D thesis, Sorbonne Universit\'e, 2018.

\bibitem{Stru2016}
  \newblock M. Strugarek, N. Vauchelet,
  \newblock  Reduction to a single closed equation for 2 by 2 reaction-diffusion systems of Lotka-Volterra type,
  \newblock \emph{SIAM J. Appl. Math.} \textbf{76} (2016) no 5, 2068--2080.

\bibitem{Thome}
    \newblock R.C.A. Thome, H.M. Yang, and L. Esteva, 
    \newblock Optimal control of Aedes aegypti mosquitoes by the sterile insect technique and insecticide,
    \newblock \emph{Math. Biosci.} \textbf{223} (2010), pp. 12--23.
    
\bibitem{Wer.Wolbachia}
  \newblock J. H. Werren, L. Baldo and M. E. Clark,
  \newblock \textit{Wolbachia}: master manipulators of invertebrate biology,
  \newblock \emph{Nature Review Microbiology} (2008) \textbf{8}, 741--751.

\bibitem{Wal.wMel}
  \newblock T. Walker, P. H. Johnson, L. A. Moreira, I. Iturbe-Ormaetxe, F. D. Frentiu, C. J. McMeniman, Y. S. Leong, Y. Dong, J. Axford, P. Kriesner, A. L. Lloyd, S. A. Ritchie, S. L. O'Neill, A. A. Hoffmann,
  \newblock The wMel \textit{Wolbachia} strain blocks dengue and invades caged \textit{Aedes aegypti} populations,
  \newblock \emph{Nature} (2011) \textbf{476}, 450--453.
  
\bibitem{Yang}
  \newblock H. Yang, M. Macoris, K. Galvani, M. Andrighetti and D. Wanderley,
  \newblock Assessing the effects of temperature on the population of Aedes aegypti, the vector of dengue,
  \newblock \emph{Epidemiol Infect}, \textbf{137} (2009), 1188--1202.

  
  
\end{thebibliography}
\end{document}